\documentclass{amsart}

\RequirePackage{fix-cm}
\RequirePackage[l2tabu, orthodox]{nag}
\RequirePackage[all, warning]{onlyamsmath}

\usepackage{latexsym}
\usepackage{amsmath,amsfonts,amssymb}
\usepackage{amsthm}
\usepackage[all]{xy}
\usepackage{graphicx,color,xcolor}
\usepackage{overpic}
\usepackage{subfig}
\usepackage{braket}
\usepackage{booktabs}
\usepackage{threeparttable}
\usepackage[bookmarks=true,bookmarksnumbered=true]{hyperref}
\usepackage{multirow,bigdelim}
\usepackage{enumerate}
\usepackage{cleveref}
\numberwithin{equation}{section}

\graphicspath{{./fig/}}
\newtheorem{theorem}{Theorem}
\newtheorem{lemma}{Lemma}

\newtheorem{proposition}{Proposition}

\theoremstyle{definition}
\newtheorem{definition}{Definition}

\newtheorem{example}{Example}
\newtheorem{remark}{Remark}

\crefname{section}{Section}{Sections}
\crefname{appendix}{Appendix}{Appendices}
\crefname{theorem}{Theorem}{Theorems}
\crefname{lemma}{Lemma}{Lemmas}
\crefname{corollary}{Corollary}	{Corollaries}
\crefname{proposition}{Proposition}{Propositions}
\crefname{claim}{Claim}{Claims}
\crefname{conjecture}{Conjecture}{Conjectures}
\crefname{definition}{Definition}{Definitions}
\crefname{problem}{Problem}{Problems}
\crefname{example}{Example}{Examples}
\crefname{remark}{Remark}{Remarks}
\crefname{figure}{Figure}{Figures}
\crefname{footnote}{Footnote}{Footnotes}
\crefname{equation}{}{}
\crefname{enumi}{}{}
\crefformat{enumi}{(#2#1#3)}
\Crefformat{enumi}{(#2#1#3)}

\newcommand{\R}{\mathbb{R}}
\newcommand{\N}{\mathbb{N}}

\newcommand{\ld}{,\ldots,}

\newcommand{\ep}{\varepsilon}
\newcommand{\wt}{\widetilde}

\newcommand{\prn}[1]{\left(#1\right)}
\newcommand{\norm}[1]{\left\|#1\right\|}
\newcommand{\abs}[1]{\left|#1\right|}

\newcommand{\D}{\displaystyle}
\newfont{\bg}{cmr9 scaled\magstep2}
\newcommand{\bigzerol}{\smash{\lower1.0ex\hbox{\bg 0}}}

\DeclareMathOperator*{\minimize}{minimize}

\begin{document}

\title[All unconstrained strongly convex problems are weakly simplicial]{All unconstrained strongly convex problems\\
are weakly simplicial}

\author{Yusuke \textsc{Mizota}}
\address{Faculty of Science and Engineering
Kyushu Sangyo University
3-1 Matsukadai, 2-Chome, Higashi-ku Fukuoka 813-8503 Japan
}
\email{mizota@ip.kyusan-u.ac.jp}

\author{Naoki \textsc{Hamada}}
\address{
Technology Expert Department,
Engineering Division,
KLab Inc.,
Tokyo 106-6122, Japan\\
RIKEN AIP-Fujitsu Collaboration Center,
RIKEN,
Tokyo 103-0027, Japan}
\email{hamada-n@klab.com}

\author{Shunsuke \textsc{Ichiki}}
\address{
Department of Mathematical and Computing Science,
School of Computing,
Tokyo Institute of Technology,
Tokyo 152-8552,
Japan}
\email{ichiki@c.titech.ac.jp}

\date{}
\begin{abstract}
A multi-objective optimization problem is $C^r$ weakly simplicial if there exists a $C^r$ surjection from a simplex onto the Pareto set/front such that the image of each subsimplex is the Pareto set/front of a subproblem, where $0\leq r\leq \infty$.
This property is helpful to compute a parametric-surface approximation of the entire Pareto set and Pareto front.
It is known that all unconstrained strongly convex $C^r$ problems are $C^{r-1}$ weakly simplicial for $1\leq r \leq \infty$.
In this paper, we show that all unconstrained strongly convex problems are $C^0$ weakly simplicial.
The usefulness of this theorem is demonstrated in a sparse modeling application: we reformulate the elastic net as a non-differentiable multi-objective strongly convex problem and approximate its Pareto set (the set of all trained models with different hyper-parameters) and Pareto front (the set of performance metrics of the trained models) by using a B\'ezier simplex fitting method, which accelerates hyper-parameter search.
\end{abstract}
\maketitle
\section{Introduction}\label{sec:intro}
Given a subset $X$ of a Euclidean space $\R^n$ and functions $f_1,\dots,f_m:X\to\R$, a multi-objective optimization problem of minimizing a mapping $f=(f_1,\dots,f_m):X\to\R^m$ is denoted by
\[
    \minimize_{x\in X} f(x)=(f_1(x),\dots,f_m(x)).
\]
The goal of this problem is to find (or approximate) the Pareto set, i.e., the set of all Pareto solutions, and the Pareto front, i.e., the image of the Pareto set under $f$.
Trade-off analysis and decision making in our real-life can be treated as a multi-objective optimization problem (e.g., manufacturing \cite{Wang2011}, optimal control \cite{Peitz2018}, and multi-agent systems \cite{Radulescu2019}).
Presenting the entire Pareto set and front rather than a single solution gives a global perspective on the trade-off of candidate solutions and enables us to make a better decision.

There is a problem class where one can efficiently compute a parametric-surface approximation of the entire Pareto set and front with a theoretical guarantee.
If a multi-objective optimization problem is $C^r$ weakly simplicial (see \cref{sec:main} for definition), then there exists a $C^r$ surjection from a simplex onto the Pareto set, as well as onto the Pareto front \cite{Hamada2019}.
Any continuous mapping from a simplex to a Euclidean space can be approximated by a B\'ezier simplex with an arbitrary accuracy \cite{Kobayashi2019}.
The asymptotic $L_2$-risk of B\'ezier simplex approximation has been evaluated \cite{Tanaka2020}.
For $1\le r\le\infty$, all unconstrained strongly convex $C^r$ problems are $C^{r-1}$ weakly simplicial, where the solution of weighted sum scalarization problems gives a $C^{r-1}$ surjection from a simplex of all possible weight vectors onto the Pareto set and front \cite{Hamada2019,Hamada2019b}.
By approximating this surjection with a B\'ezier simplex, one can obtain a parametric-surface approximation of the Pareto set and front of any $C^1$ strongly convex problem.

However, there are many important problems that are strongly convex but non-differentiable.
One motivating example appears in the context of sparse modeling.
The elastic net \cite{Zou2005} is a linear regression method that minimizes a weighted sum of the ordinary least square error, the $L_1$-regularization, and the $L_2$-regularization.
Its hyper-parameter search, i.e., finding the preferred model among all candidate models trained with different weighting coefficients, can be written as a strongly convex non-differentiable multi-objective optimization problem.
If this problem is weakly simplicial, then a B\'ezier simplex can be built from training results of the elastic net in order to approximate the correspondence from a simplex of all hyper-parameters onto the Pareto set of all trained models and the Pareto front of all performance metrics.
Such an approximation not only reduces the number of training trials of the elastic net but also provides geometric insights to select the preferred model.

In this paper, we show all unconstrained strongly convex problems are $C^0$ weakly simplicial.
In \cref{sec:main}, we present some definitions and the main result (\cref{thm:main}).
By lemmas prepared in \cref{sec:pre}, we prove \cref{thm:main} in \cref{sec:mainproof}.
In \cref{sec:application}, \cref{thm:main} is applied to a hyper-parameter search problem of the elastic net: multi-objective formulation of the problem and B\'ezier simplex approximation are presented with numerical simulation.
We conclude the paper in \cref{sec:conclusions}.

\section{Main result}\label{sec:main}
Throughout this paper, $m$ and $n$ are positive integers, and we denote the index set $\set{1\ld m}$ by $M$.

We consider the problem of optimizing several functions simultaneously.
More precisely, let $f: X \to \R^m$ be a mapping, where $X$ is a given arbitrary set.
A point $x \in X$ is called a \emph{Pareto solution} of $f$ if there does not exist another point $y \in X$ such that $f_i(y) \leq f_i(x)$ for all $i \in M$ and $f_j(y) < f_j(x)$ for at least one index $j \in M$.
We denote the set consisting of all Pareto solutions of $f$ by $X^*(f)$, which is called the \emph{Pareto set} of $f$.
The set $f(X^*(f))$ is called the \emph{Pareto front} of $f$.
The problem of determining $X^*(f)$ is called the \emph{problem of minimizing $f$}.

Let $f=(f_1\ld f_m): X \to \R^m$ be a mapping, where $X$ is a given arbitrary set.
For a non-empty subset $I=\set{i_1\ld i_k}$ of $M$ such that $i_1 < \dots < i_k$, set
\begin{align*}
    f_I=(f_{i_1}\ld f_{i_k}).
\end{align*}
The problem of determining $X^*(f_I)$ is called a \emph{subproblem} of the problem of minimizing $f$.
Set
\begin{align*}
    \Delta^{m - 1} &= \Set{(w_1, \dots, w_m) \in \R^m | \sum_{i = 1}^m w_i = 1,\ w_i \geq 0}.
\end{align*}
We also denote a face of $\Delta^{m - 1}$ for a non-empty subset $I$ of $M$ by
\begin{align*}
    \Delta_I = \set{(w_1, \dots, w_m) \in \Delta^{m - 1} | w_i = 0\ (i \not \in I)}.
\end{align*}
In this paper, $r$ is a non-negative integer or $r=\infty$. 
For a $C^r$ manifold $N$ (possibly with corners) and a subset $V$ of $\R^\ell$, a mapping $g:N\to V$ is called a \emph{$C^r$ mapping} (resp., a  \emph{$C^r$ diffeomorphism}) if $g:N\to \R^\ell$ is of class $C^r$ (resp.,  $g:N\to \R^\ell$ is a $C^r$ immersion and $g:N\to V$ is a homeomorphism), where $r\geq 1$.
In this paper, $C^0$ mappings and $C^0$ diffeomorphisms are continuous mappings and homeomorphisms, respectively.

By referring to \cite{Hamada2019}, we give the definition of (weakly) simplicial problems in this paper.
\begin{definition}\label{def:simplicial}
    Let $f = (f_1\ld f_m): X \to \R^m$ be a mapping, where $X$ is a subset of $\R^n$.
    The problem of minimizing $f$ is $C^r$ \emph{simplicial} if there exists a $C^r$ mapping $\Phi: \Delta^{m - 1} \to X^*(f)$ such that both the mappings $\Phi|_{\Delta_I}: \Delta_I \to X^*(f_I)$ and $f|_{X^*(f_I)}: X^*(f_I) \to f(X^*(f_I))$ are $C^r$ diffeomorphisms for any non-empty subset $I$ of $M$, where $0\leq r\leq \infty$.
    The problem of minimizing $f$ is $C^r$ \emph{weakly simplicial}\footnote{
        In \cite{Hamada2019}, the problem of minimizing $f:X\to \R^m$ is said to be $C^r$ \emph{weakly simplicial} if there exists a $C^r$ mapping $\phi: \Delta^{m - 1} \to f(X^*(f))$ satisfying $\phi(\Delta_I) = f(X^*(f_I))$ for any non-empty subset $I$ of $M$.
        On the other hand, a surjective mapping of $\Delta^{m-1}$ into $X^*(f)$ is important to describe $X^*(f)$.
        Hence, the definition of weak simpliciality in this paper is updated from that in \cite{Hamada2019}.
        \label{ftn:weak}
    } if there exists a $C^r$ mapping $\phi: \Delta^{m - 1} \to X^*(f)$ such that $\phi(\Delta_I) = X^*(f_I)$ for any non-empty subset $I$ of $M$,  where $0\leq r\leq \infty$.
\end{definition}

A subset $X$ of $\R^n$ is \emph{convex} if $t x + (1 - t) y \in X$ for all $x, y \in X$ and all $t \in [0, 1]$.
Let $X$ be a convex set in $\R^n$.
A function $f: X \to \R$ is \emph{strongly convex} if there exists $\alpha > 0$ such that
\begin{align*}
    f(t x + (1 - t) y) \leq t f(x) + (1 - t) f(y) - \frac{1}{2} \alpha t (1 - t) \norm{x - y}^2
\end{align*}
for all $x, y \in X$ and all $t \in [0, 1]$, where $\norm{z}$ is the Euclidean norm of $z \in \R^n$.
The constant $\alpha$ is called a \emph{convexity parameter} of the function $f$.
A mapping $f = (f_1\ld f_m): X \to \R^m$ is \emph{strongly convex} if $f_i$ is strongly convex for any $i \in M$.
The problem of minimizing a strongly convex $C^r$ mapping is called the \emph{strongly convex $C^r$ problem}.

\begin{theorem}[\cite{Hamada2019,Hamada2019b}]\label{thm:main_d}
    Let $f: \R^n \to \R^m$ be a strongly convex $C^r$ mapping, where $1\leq r\leq \infty$.
   Then, the problem of minimizing $f$ is $C^{r-1}$ weakly simplicial.
\end{theorem}
As a practical application described in this paper, it is sufficient to consider weak simpliciality (for details, see \cref{sec:application}).
For the readers' convenience, we give the following remark on simpliciality.
\begin{remark}
    In \cref{thm:main_d}, if the rank of the differential $df_x$ is equal to $m - 1$ for any $x \in X^*(f)$, then the problem of minimizing $f$ is 
    $C^{r-1}$ simplicial (for details, see \cite{Hamada2019,Hamada2019b}).
\end{remark}
As in \cite{Hamada2019,Hamada2019b}, the differentiability is essential for the proof of \cref{thm:main_d}.
Namely, the method of the proof of \cref{thm:main_d} does not work if we omit the differentiability.
However, the following theorem asserts that all unconstrained strongly convex problems are $C^0$ weakly simplicial.
\begin{theorem}\label{thm:main}
    Let $f: \R^n \to \R^m$ be a strongly convex mapping.
    Then, the problem of minimizing $f$ is $C^0$ weakly simplicial.
\end{theorem}
\begin{remark}\label{rem:main}
    Note that (strict) convexity of a mapping does not necessarily imply that the problem is $C^0$ weakly simplicial.
    For example, the problem of minimizing $f:\R\to \R$ defined by $f(x)=e^x$ does not have a Pareto solution (i.e.~a minimizer).
Thus, it is not $C^0$ weakly simplicial although $f$ is strictly convex.
\end{remark}

\section{Preliminaries for the proof of \texorpdfstring{\cref{thm:main}}{Theorem 2}}\label{sec:pre}
In this section, we prepare some lemmas for the proof of \cref{thm:main}.
The following lemma follows from \cite[Theorem~3.1.8 (p.~121)]{Nesterov2004}.
\begin{lemma}\label{thm:convexconti}
    If a function $f: \R^n \to \R$ is strongly convex, then $f$ is continuous.
\end{lemma}

We give the following two lemmas (\cref{thm:sufficient,thm:necessary}) in \cite{Miettinen1999}.
\begin{lemma}[{\cite[Theorem~3.1.3~in~Part~I\hspace{-.1em}I~(p.~79)]{Miettinen1999}}]\label{thm:sufficient}
    Let $f = (f_1, \dots, f_m): \R^n \to \R^m$ be a mapping and let $(w_1, \dots, w_m) \in \Delta^{m - 1}$.
    If $x \in \R^n$ is the unique minimizer of the function $\sum_{i = 1}^m w_i f_i$, then $x \in X^*(f)$.
\end{lemma}
Let $X$ be a convex subset of $\R^n$.
A function $f: X \to \R$ is said to be \emph{convex} if
\begin{align*}
    f(t x + (1 - t) y) \leq t f(x) + (1 - t) f(y)
\end{align*}
for all $x, y \in X$ and all $t \in [0, 1]$.
A mapping $f = (f_1\ld f_m): X \to \R^m$ is \emph{convex} if $f_i$ is convex for any $i \in M$.
\begin{lemma}[{\cite[Theorem~3.1.4~in~Part~I\hspace{-.1em}I~(p.~79)]{Miettinen1999}}]\label{thm:necessary}
    Let $f = (f_1, \dots, f_m): \R^n \to \R^m$ be a convex mapping.
    If $x \in X^*(f)$, then there exists some $(w_1, \dots, w_m) \in \Delta^{m - 1}$ such that $x$ is a minimizer of $\sum_{i = 1}^m w_i f_i$.
\end{lemma}

Now, we prepare the following three lemmas (\cref{thm:compa,thm:minimum,thm:inequ}) on strongly convex functions without differentiability.
\begin{lemma}\label{thm:compa}
    Let $f: \R^n \to \R$ be a strongly convex function. Then, $f^{-1}((-\infty, f(0)])$ is compact, where $0$ is the origin of $\R^n$ and $(-\infty, f(0)]=\set{y\in\R|y\leq f(0)}$.
\end{lemma}

\begin{proof}[Proof of \cref{thm:compa}]
    Since $f$ is continuous by \cref{thm:convexconti} and $(-\infty, f(0)]$ is closed, $f^{-1}((-\infty, f(0)])$ is also closed.
    Hence, it is sufficient to show that $f^{-1}((-\infty, f(0)])$ is bounded.
    Let $x\in f^{-1}((-\infty, f(0)])$ be an arbitrary element.
    Then, there exist some real number $\lambda\geq 0$ and some $v\in S^{n-1}$ such that $x=\lambda v$, where $S^{n-1}$ is the  $(n-1)$-dimensional unit sphere centered at the origin.
    In the case $0\leq \lambda \leq 1$, we have $\norm{x}\leq 1$.
    Now, we consider the case $\lambda>1$. Since $f$ is strongly convex, there exists $\alpha>0$ such that
    \begin{align*}
        f(t\cdot 0 + (1 - t) x) \leq t f(0) + (1 - t) f(x) - \frac{1}{2} \alpha t (1 - t) \norm{0 - x}^2
    \end{align*}
    for all $t \in [0, 1]$. 
    Since $0<\dfrac{\lambda-1}{\lambda}<1$, by substituting $t=\dfrac{\lambda-1}{\lambda}$ into the above inequality, we have
    \begin{align*}
        f(v) &\leq \frac{\lambda-1}{\lambda} f(0) + \frac{1}{\lambda} f(x) - \frac{\alpha}{2}\frac{\lambda-1}{\lambda^2} \lambda^2\norm{v}^2.
    \end{align*}
    Since $f(x)\leq f(0)$ and $\norm{v}=1$, we obtain
    \begin{align*}
        \lambda &\leq  \frac{2f(0)-2f(v)}{\alpha}+1.
    \end{align*}
    Since $f$ is continuous by \cref{thm:convexconti}, $S^{n-1}$ is compact and $\norm{x}=\lambda$, we have
    \begin{align*}
        \norm{x} &\leq  \frac{2f(0)-2c}{\alpha}+1,
    \end{align*}
    where
    $\displaystyle c=\min_{v\in S^{n-1}}f(v)$.
    Therefore, it follows that 
    \begin{align*}
        \norm{x}\leq \max\Set{1,\frac{2f(0)-2c}{\alpha}+1}
    \end{align*}
    for any $x \in f^{-1}((-\infty, f(0)])$.
    Hence, $f^{-1}((-\infty, f(0)])$ is bounded.
\end{proof}

\begin{lemma}\label{thm:minimum}
    A strongly convex function $f: \R^n \to \R$ has a unique minimizer.
\end{lemma}
\begin{proof}[Proof of \cref{thm:minimum}]
    Since $f^{-1}((-\infty, f(0)])$ $(\not=\emptyset)$ is compact by \cref{thm:compa} and $f$ is continuous by \cref{thm:convexconti}, $f$ has a minimizer in $f^{-1}((-\infty, f(0)])$.

    We assume that $x\neq y$ for some minimizers $x$ and $y$. Then, since $f$ is strongly convex, there exists $\alpha > 0$ such that
    \begin{align*}
        f(t x + (1 - t) y)\leq t f(x)+ (1-t) f(y)-\frac{1}{2} \alpha t(1-t) \norm{x-y}^2 
    \end{align*}
    for some $t \in (0, 1)$.
    Since $f(x)=f(y)$ and $\norm{x-y}>0$, we have that
    \begin{align*}
        f(tx+(1-t)y)<f(x)
    \end{align*}
    for some $t\in (0, 1)$.
    This contradicts the fact that $x$ is a minimizer of $f$.
    Thus, a minimizer of $f$ must be unique.
\end{proof}

\begin{lemma}\label{thm:inequ}
    Let $f:\R^n\to \R$ be a strongly convex function with a convexity parameter $\alpha >0$ and $x_0$ be the unique minimizer of $f$.
    Then, we have
    \begin{align*}
        f(x_0) + \frac{\alpha}{2} \norm{x - x_0}^2 \leq f(x)
    \end{align*}
    for any $x \in \R^n$.
\end{lemma}

\begin{proof}[Proof of \cref{thm:inequ}]
    We assume that there exists $x_1\in\R^n$ such that 
    \begin{align*}
        f(x_0) + \frac{\alpha}{2} \norm{x_1 - x_0}^2 > f(x_1).
    \end{align*}
    Set $\delta = f(x_1)-f(x_0) -\frac{\alpha}{2} \norm{x_1 - x_0}^2$. 
    Note that $\delta < 0$ and $x_1\neq x_0$.  
    Since $x_0$ is a minimizer of $f$ and the function $f$ is strongly convex, we have 
     \begin{align*}
        f(x_0) \leq  f(t x_1 + (1 - t) x_0)\leq f(x_0)+\delta t+\frac{1}{2} \alpha t^2 \norm{x_1 - x_0}^2
    \end{align*}
    for all $t \in (0, 1]$.
    By dividing the above inequality by $t\in (0,1]$, we obtain
    \begin{align*}
        -\delta \leq \frac{1}{2} \alpha t \norm{x_1 - x_0}^2 
    \end{align*}
    for all $t \in (0, 1]$.
    Since 
     \begin{align*}
     \lim_{t \to 0}\frac{1}{2} \alpha t \norm{x_1 - x_0}^2=0,
     \end{align*}
     this contradicts $-\delta>0$.
\end{proof}
By the same method as in the proof of \cite[Lemma~2.1.4~(p.~64)]{Nesterov2004}, we have the following.
\begin{lemma}\label{thm:weight}
    Let $f_i: \R^n \to \R$ be a strongly convex function with a convexity parameter $\alpha_i > 0$, where $i\in M$.
    Then, for any $(w_1, \dots, w_m) \in \Delta^{m - 1}$, the function $\sum_{i = 1}^m w_i f_i: \R^n \to \R$ is a strongly convex function with a convexity parameter $\sum_{i = 1}^m w_i \alpha_i$.
\end{lemma}
\begin{proof}[Proof of \cref{thm:weight}]
    Since $f_i$ is a strongly convex function with a convexity parameter $\alpha_i > 0$  for any $i \in M$, we have that
    \begin{align*}
        f_i(t x + (1 - t) y) \leq t f_i(x) + (1 - t) f_i(y) - \frac{1}{2} \alpha_i t (1 - t) \norm{x - y}^2
    \end{align*}
    for all $x, y \in \R^n$ and all $t \in [0, 1]$. 
    Since $w_i\geq 0$ for any $i \in M$, we have the following inequality
    \begin{align*}
        &\left(\sum_{i=1}^{m} w_if_i\right)(t x + (1 - t) y) \\
        \leq & t \left(\sum_{i=1}^{m} w_if_i\right)(x) + (1 - t) \left(\sum_{i=1}^{m} w_if_i\right)(y) - \frac{1}{2} \left(\sum_{i=1}^{m} w_i\alpha_i\right) t (1 - t) \norm{x - y}^2
    \end{align*} 
    for all $x, y \in \R^n$ and all $t \in [0, 1]$.
    Since $\sum_{i = 1}^m w_i = 1$, $\sum_{i = 1}^m w_i \alpha_i$ is positive.
    Thus, $\sum_{i = 1}^m w_i f_i: \R^n \to \R$ is a strongly convex function with a convexity parameter $\sum_{i = 1}^m w_i \alpha_i$.
\end{proof}

In order to give the last lemma (\cref{thm:ine}) in this section, which is essentially used in the proof of \cref{thm:main}, we prepare the following three lemmas (\cref{thm:weak,thm:closed,thm:compact}). 

Let $f: X \to \R^m$ be a mapping, where $X$ is a given arbitrary set.
A point $x \in X$ is called a \emph{weak Pareto solution} of $f$ if there does not exist another point $y \in X$ such that $f_i(y) < f_i(x)$ for all $i \in M$.
Then, by $X^{\mathrm w}(f)$, we denote the set consisting of all weak Pareto solutions of $f$.

\begin{lemma}[\cite{Hamada2019c}]\label{thm:weak}
    Let $f: \R^n \to \R^m$ be a strongly convex mapping.
    Then, we have $X^*(f) = X^{\mathrm w}(f)$.
\end{lemma}

\begin{lemma}[\cite{Hamada2019b}]\label{thm:closed}
    Let $f: X \to \R^m$ be a continuous mapping, where $X$ is a topological space.
    Then, $X^{\mathrm w}(f)$ is a closed set of $X$.
\end{lemma}

\begin{lemma}\label{thm:compact}
    Let $f: \R^n \to \R^m$ be a strongly convex mapping.
    Then, $X^*(f)$ is compact.
\end{lemma}

\begin{proof}[Proof of \cref{thm:compact}]
    By \cref{thm:closed,thm:weak}, it follows that $X^*(f)$ is closed.
    Thus, for the proof, it is sufficient to show that $X^*(f)$ is bounded.
    Let $\alpha_i > 0$ be a convexity parameter of $f_i$, where $f = (f_1\ld f_m)$ and $i \in M$.
    By \cref{thm:minimum}, $f_i$ has a unique minimizer $x_i\in \R^n$ for any $i \in M$.
    Set
    \begin{align*}
        \Omega_i = \Set{x \in \R^n | f_i(x_i) + \frac{\alpha_i}{2} \norm{x - x_i}^2 \leq f_i(x_1)}.
    \end{align*}
    Since every $\Omega_i$ is compact, $\Omega = \bigcup_{i = 1}^m \Omega_i$ is also compact.
    Hence, in order to show that $X^*(f)$ is bounded, it is sufficient to show that $X^*(f) \subseteq \Omega$.
    Suppose that there exists an element $x' \in X^*(f)$ such that $x' \not \in \Omega$.
    Then, it follows that
    \begin{align}\label{eq:com}
        f_i(x_i) + \frac{\alpha_i}{2} \norm{x' - x_i}^2 > f_i(x_1)
    \end{align}
    for any $i \in M$.
    By \cref{thm:inequ}, we have
    \begin{align}\label{eq:com-2}
        f_i(x_i) + \frac{\alpha_i}{2} \norm{x' - x_i}^2 \leq f_i(x')
    \end{align}
      for any $i \in M$.
    From \cref{eq:com,eq:com-2}, it follows that $f_i(x') > f_i(x_1)$ for any $i \in M$.
    This contradicts $x' \in X^*(f)$.
\end{proof}
Now, we can define a mapping from $\Delta^{m - 1}$ into $X^*(f)$ for any strongly convex mapping $f:\R^m\to \R^n$.
Let $w = (w_1\ld w_m) \in \Delta^{m - 1}$.
Since $\sum_{i = 1}^m w_i f_i: \R^n \to \R$ is a strongly convex function by \cref{thm:weight}, the function $\sum_{i = 1}^m w_i f_i$ has a unique minimizer by \cref{thm:minimum}.
By \cref{thm:sufficient}, this minimizer is contained in $X^*(f)$.
Hence, we can define a mapping $x^*: \Delta^{m - 1} \to X^*(f)$ as follows:
\begin{align}\label{eq:map}
    x^*(w) = \arg \min_{x \in \R^n} \prn{\sum_{i = 1}^m w_i f_i(x)},
\end{align}
where $\arg \min_{x \in \R^n} \prn{\sum_{i = 1}^m w_i f_i(x)}$ is the unique minimizer of $\sum_{i = 1}^m w_i f_i$.
Note that in \cite{Hamada2019,Hamada2019b}, $x^*$ is defined only in the case where a given strongly convex mapping $f:\R^n\to \R^m$ is of class at least $C^1$.
On the other hand, by lemmas of this paper, it can be defined for any strongly convex mapping $f:\R^m\to \R^n$. 

We can show the following lemma by the same argument as in the proof of \cite[Lemma~12]{Hamada2019b} which works only in the case where a strongly convex mapping is of class $C^1$.
On the other hand, the following lemma works without differentiability.  
\begin{lemma}\label{thm:ine}
    Let $f = (f_1, \dots, f_m): \R^n \to \R^m$ be a strongly convex mapping.
    Let $\alpha_i > 0$ be a convexity parameter of $f_i$ and $K_i$ be the maximal value of $F_i: X^*(f) \times X^*(f) \to \R$ defined by $F_i(x, y) = \abs{f_i(x) - f_i(y)}$ for any $i \in M$.
    Then, for any $w = (w_1\ld w_m), \wt{w} = (\wt{w}_1\ld \wt{w}_m) \in \Delta^{m - 1}$, we have that
    \begin{align*}
        \norm{x^*(w) - x^*(\wt{w})} \leq \sqrt{\D \frac{K_0}{\alpha_0} \sum_{i = 1}^m \abs{w_i - \wt{w}_i}},
    \end{align*}
    where $\alpha_0 = \min \set{\alpha_1\ld \alpha_m}$ and $K_0 = \max \set{K_1\ld K_m}$.
\end{lemma}
\begin{remark}
    In \cref{thm:ine}, the Pareto set $X^*(f)$ is compact by \cref{thm:compact}.
    For any $i \in M$, since $f_i$ is continuous by \cref{thm:convexconti}, the function $F_i$ has the maximal value $K_i$.
\end{remark}

\begin{proof}[Proof of \cref{thm:ine}]
    Let $w, \wt{w} \in \Delta^{m - 1}$ be arbitrary elements.
    By \cref{thm:weight}, the function $\sum_{i = 1}^m w_i f_i: \R^n \to \R$ (resp., $\sum_{i = 1}^m \wt{w}_i f_i: \R^n \to \R$) is a strongly convex function with a convexity parameter $\sum_{i = 1}^m w_i \alpha_i$ (resp., $\sum_{i = 1}^m \wt{w}_i \alpha_i$).
    Since $x^*(w)$ (resp., $x^*(\wt{w})$) is the minimizer of $\sum_{i = 1}^m w_i f_i$ (resp., $\sum_{i = 1}^m \wt{w}_i f_i$), by \cref{thm:inequ}, we get
    {\small
    \begin{align}\label{eq:ine-1}
        \prn{\sum_{i = 1}^m w_i f_i}(x^*(w)) + \frac{\sum_{i = 1}^m w_i \alpha_i}{2}
        \norm{x^*(\wt{w}) - x^*(w)}^2
        &\leq
        \prn{\sum_{i = 1}^m w_i f_i}(x^*(\wt{w})),
        \end{align}
        \begin{align}\label{eq:ine-2}
        \prn{\sum_{i = 1}^m \wt{w}_i f_i}(x^*(\wt{w})) + \frac{\sum_{i = 1}^m \wt{w}_i \alpha_i}{2}
        \norm{x^*(w) - x^*(\wt{w})}^2
        &\leq
        \prn{\sum_{i = 1}^m \wt{w}_i f_i}(x^*(w)).
    \end{align}
    }By \cref{eq:ine-1,eq:ine-2}, we have
    {\small
    \begin{align}\label{eq:ine-3}
        \frac{\sum_{i = 1}^m w_i \alpha_i}{2}
        \norm{x^*(\wt{w}) - x^*(w)}^2
        &\leq
        \sum_{i = 1}^m w_i \prn{f_i(x^*(\wt{w})) - f_i(x^*(w))},
    \end{align}
    \begin{align}\label{eq:ine-4}
        \frac{\sum_{i = 1}^m \wt{w}_i \alpha_i}{2}
        \norm{x^*(\wt{w}) - x^*(w)}^2
        &\leq
        \sum_{i = 1}^m \wt{w}_i \prn{f_i(x^*(w)) - f_i(x^*(\wt{w}))},
    \end{align}
    }respectively.
    By \cref{eq:ine-3,eq:ine-4}, we obtain
    {\small
    \begin{align*}
        \frac{\sum_{i = 1}^m (w_i + \wt{w}_i) \alpha_i}{2}
        \norm{x^*(\wt{w}) - x^*(w)}^2
        &\leq
        \sum_{i = 1}^m (w_i - \wt{w}_i)
        (
        f_i(x^*(\wt{w})) - f_i(x^*(w))
        ).
    \end{align*}
    }By the inequality above and $\sum_{i = 1}^m(w_i + \wt{w}_i) = 2$, it follows that 
    {\small
    \begin{align}\label{eq:ine-5}
        \alpha_0
        \norm{x^*(\wt{w}) - x^*(w)}^2
        &\leq
        \sum_{i = 1}^m (w_i - \wt{w}_i)
        (
        f_i(x^*(\wt{w})) - f_i(x^*(w))
        ).
    \end{align}
    }We also obtain
    {\small
    \begin{align*}
        \sum_{i = 1}^m \prn{w_i - \wt{w}_i}
        \prn{f_i(x^*(\wt{w})) - f_i(x^*(w))}
        &\leq
        \sum_{i = 1}^m \abs{w_i - \wt{w}_i}
        \abs{f_i(x^*(\wt{w})) - f_i(x^*(w))}
        \\
        &\leq
        \sum_{i = 1}^m \abs{w_i - \wt{w}_i} K_i.
        \\
        &\leq
        K_0 \sum_{i = 1}^m \abs{w_i - \wt{w}_i}.
    \end{align*}
    }By the inequality above and \cref{eq:ine-5}, we have
    \begin{align*}
        \alpha_0 \norm{x^*(w) - x^*(\wt{w})}^2
        \leq
        K_0 \sum_{i = 1}^m \abs{w_i - \wt{w}_i}.
    \end{align*}
    Therefore, it follows that
    \begin{align*}
        \norm{x^*(w) - x^*(\wt{w})} \leq
        \sqrt{
        \D \frac{K_0}{\alpha_0} \sum_{i = 1}^m \abs{w_i - \wt{w}_i}
        }.
    \end{align*}
\end{proof}

\section{Proof of \texorpdfstring{\cref{thm:main}}{Theorem 2}}\label{sec:mainproof}
First, we give the following essential result for the proof of \cref{thm:main} (for the definition of $x^*:\Delta^{m-1}\to X^*(f)$ in \cref{thm:contisurj}, see \cref{eq:map}).
\begin{proposition}\label{thm:contisurj}
    Let $f:\R^n \to \R^m$ be a strongly convex mapping.
    Then, the mapping $x^*: \Delta^{m - 1} \to X^*(f)$ is surjective and continuous.
\end{proposition}

\cref{thm:main} follows from \cref{thm:contisurj} as follows:
Let $I=\set{i_1\ld i_k}$ $(i_1<\cdots <i_k)$ be an arbitrary non-empty subset of $M$ as in \cref{sec:main}.
Since $f_I: \R^n \to \R^k$ is a strongly convex mapping, $x^*|_{\Delta_I}: \Delta_I \to X^*(f_I)$ is surjective and continuous by \cref{thm:contisurj}.
Hence, the problem of minimizing $f$ is $C^0$ weakly simplicial.

By the argument above, in order to complete the proof of \cref{thm:main}, it is sufficient to show \cref{thm:contisurj}.
\begin{proof}[Proof of \cref{thm:contisurj}]
    By \cref{thm:necessary}, the mapping $x^*$ is surjective.
    Now, we prove that $x^*$ is continuous.
    Let $\wt{w} = (\wt{w}_1\ld \wt{w}_m) \in \Delta^{m - 1}$ be any element.
    For the proof, it is sufficient to show that $x^*$ is continuous at $\wt{w}$.
    Let $\ep$ be an arbitrary positive real number.
    Then, there exists an open neighborhood $V$ of $\wt{w}$ in $\Delta^{m - 1}$ satisfying
    \begin{align*}
        \sqrt{\D \frac{K_0}{\alpha_0} \sum_{i = 1}^m \abs{w_i - \wt{w}_i}} < \ep
    \end{align*}
    for any $w \in V$, where $K_0$ and $\alpha_0$ are defined in \cref{thm:ine}.
    From \cref{thm:ine}, it follows that
    \begin{align*}
        \norm{x^*(w) - x^*(\wt{w})} < \ep
    \end{align*}
    for any $w \in V$.
\end{proof}

\begin{remark}
    The mapping $x^*$ in \cref{thm:contisurj} is not necessarily injective even in the case where the minimizers of $f_1\ld f_m$ are different as follows.
    Let $f = (f_1, f_2): \R \to \R^2$ be the mapping defined by
    \begin{align*}
        f_1(x) & = x^2+\abs{x}, \\
        f_2(x) & = (x-1)^2+\abs{x-1}.
    \end{align*}
    Then, $f$ is strongly convex\footnote{
    It follows from the fact that a function $f: \R^n \to \R$ is strongly convex with a convexity parameter $\alpha > 0$ if and only if the function $g: \R^n \to \R$ defined by $g(x) = f(x) - \frac{\alpha}{2} \norm{x}^2$ is convex (for the proof of the fact, see for example \cite{Hamada2019b}).}
    and non-differentiable.
    Since the minimizer of $f_1$ is $x=0$ and that of $f_2$ is $x=1$, these minimizers are different. 
    Let $\varphi: [0, 1] \to \Delta^1$ be the diffeomorphism defined by $\varphi(w_1) = (w_1, 1 - w_1)$.
    We can easily obtain the following:
    \begin{align*}
        x^* \circ \varphi(w_1) & =
        \begin{cases}
            \D 1 & \text{if $0 \leq w_1 < \frac{1}{4}$},\\
            \D \frac{3-4w_1}{2} & \text{if $\frac{1}{4} \leq w_1 \leq \ \frac{3}{4}$},\\
            \D 0 & \text{if $\frac{3}{4} < w_1 \leq 1$}.\\
        \end{cases}
    \end{align*}
    Thus, $x^*$ is not injective. 
\end{remark}

\section{Application to elastic net}\label{sec:application}
In this section, we demonstrate an application of \cref{thm:main} to a statistical modeling problem.
First, we introduce the elastic net and reformulate it to a multi-objective optimization problem in \cref{sec:elastic-net}.
Second, we introduce the B\'ezier simplex and its fitting algorithm in \cref{sec:bezier-simplex}.
Through experiments, we show the practicality of this method in \cref{sec:experiments}.
Results and discussion are presented in \cref{sec:results,sec:discussion}.

\subsection{Elastic net and its multi-objective reformulation}\label{sec:elastic-net}
The elastic net~\cite{Zou2005} is a sparse modeling method that is originally a single-objective problem but can be reformulated as a multi-objective one.
Let us consider a linear regression model:
\begin{align*}
    y = \theta_1 x_1 + \theta_2 x_2 + \dots + \theta_n x_n + \zeta,
\end{align*}
where $x_i$ and $\theta_i$ $(i=1, \dots, n)$ are a predictor and its coefficient, $y$ is a response to be predicted, and $\zeta$ is a Gaussian noise.
Given a matrix $X$ with $m$ rows of observations and $n$ columns of predictors, a row vector $y$ of $m$ responses, the (original) elastic net regressor is the solution to the following problem:
\begin{equation}\label{eqn:elastic-net}
    \minimize_{\theta \in \R^n} g_{\mu,\lambda}(\theta) = \frac{1}{2m} \norm{X \theta - y}^2 + \mu \abs{\theta} + \frac{\lambda}{2} \norm{\theta}^2,
\end{equation}
where $\norm{\cdot}$ is the $\ell_2$-norm, $\abs{\cdot}$ is the $\ell_1$-norm, and $\mu$, $\lambda$ are fixed non-negative numbers for regularization.
Note that with $\mu = \lambda = 0$, the problem \cref{eqn:elastic-net} reduces to the ordinary least squares (OLS) regression; with $\mu > 0$ and $\lambda = 0$, it turns into the lasso regression \cite{Tibshirani1996}, which find a sparse solution that suppresses ineffective predictors; with $\mu = 0$ and $\lambda > 0$, it becomes the ridge regression \cite{Hoerl1970}, which finds a stable solution against multicollinear predictors.
Thus the elastic net regression, with $\mu > 0$ and $\lambda > 0$, inherits both of the lasso and ridge properties.
Choosing appropriate values for $\mu$ and $\lambda$ involves a 2-D black-box search on an unbounded domain, which often requires a great deal of computational effort.

In order to avoid such an expensive hyper-parameter search, we reformulate the problem into a multi-objective strongly convex one: first, consider its weighting problem; next, make an approximation of the solution mapping (which requires fewer models to train than the original hyper-parameter search does); then, compare possible models on the approximation and find the best weight (which is computationally cheap and does not require additional training); and finally, send the best weight back to a hyper-parameter in the original problem.

For this purpose, we separate the OLS term and the regularization terms into individual objective functions:
\[
    f_1(\theta) = \frac{1}{2m} \norm{X \theta - y}^2,\quad
    f_2(\theta) = \abs{\theta},\quad
    f_3(\theta) = \frac{1}{2}\norm{\theta}^2.
\]
The functions $f_1$ and $f_2$ are convex but may not be strongly convex.
We add a small amount of $f_3$ values to each function, making them strongly convex:
\begin{equation}\label{eqn:elastic-net-mop}
    \begin{split}
        \minimize_{\theta \in \R^n}\ & \tilde f(\theta) = (\tilde f_1(\theta), \tilde f_2(\theta), \tilde f_3(\theta))\\
        \text{where }               & \tilde f_i(\theta) = f_i(\theta) + \varepsilon f_3(\theta) \quad (i = 1, 2, 3).
    \end{split}
\end{equation}
In \cref{eqn:elastic-net-mop}, we assume that $\varepsilon$ is a positive real number.
Hence, the mapping $\tilde f$ in \cref{eqn:elastic-net-mop} is strongly convex but non-differentiable.

Now let us consider how to obtain the whole Pareto set and Pareto front of the problem \cref{eqn:elastic-net-mop}.
By \cref{thm:minimum,thm:weight}, the weighting problem
\begin{equation*}\label{eqn:elastic-net-sop}
    \minimize_{\theta \in \R^n} h_w(\theta) = w_1 \tilde f_1(\theta) + w_2 \tilde f_2(\theta) + w_3 \tilde f_3(\theta)
\end{equation*}
has a unique solution for every weight $w = (w_1, w_2, w_3) \in \Delta^2$.
We denote this solution by $\arg\min_{\theta \in \R^n} h_w(\theta)$.
By \cref{thm:contisurj}, one can define a continuous surjection $\theta^*: \Delta^2 \to X^*(\tilde f)$ by
\begin{align}
    \theta^*(w) = \arg\min_{\theta \in \R^n} h_w(\theta).\label{eqn:elastic-net-solution}
\end{align}
Since $\tilde f: \R^n\to\R^3$ is continuous by \cref{thm:convexconti}, the mapping
\begin{align*}
    \tilde f \circ \theta^*: \Delta^2 \to \tilde f(X^*(\tilde f))
\end{align*}
is also surjective and continuous, where the topology of $\tilde f(X^*(\tilde f))$ is induced from $\R^3$.
We define the \emph{Pareto graph} of $\tilde f$ by $G^*(\tilde f)= \set{(\theta,\tilde f(\theta))\in\R^{n+3} | \theta\in X^*(\tilde f)}$ and the \emph{solution mapping} by
\begin{align}
    (\theta^*, \tilde f \circ \theta^*): \Delta^2 \to G^*(\tilde f),\label{eq:solution-mapping}
\end{align}
which is again surjective and continuous, where the topology of $G^*(\tilde f)$ is induced from $\R^{n+3}$.
Since $\theta^*(\Delta^2_I) = X^*(\tilde f_I)$ for all $I$ satisfying $\emptyset \neq I \subseteq \set{1, 2, 3}$, the solution mapping contains the information of the Pareto set and the Pareto front of every subproblem, i.e., the mappings
\begin{align*}
                     \theta^*|_{\Delta^2_I}&: \Delta^2_I \to X^*(\tilde f_I),\\
    \tilde f_I \circ \theta^*|_{\Delta^2_I}&: \Delta^2_I \to \tilde f_I(X^*(\tilde f_I))
\end{align*}
for all $I$ such that $\emptyset \neq I \subseteq \set{1, 2, 3}$ are surjective and continuous, where the topology of $\Delta^2_I$,  $X^*(\tilde f_I)$, and $\tilde f_I(X^*(\tilde f_I))$ are induced from $\Delta^2$, $\R^n$, and $\R^{|I|}$, respectively.

Note that for any $w = (w_1, w_2, w_3) \in \Delta^2 \setminus \Delta^2_{\set{2, 3}}$, the point $\theta^*(w)$ is the minimizer of the function $g_{\mu(w), \lambda(w)}$ in \cref{eqn:elastic-net}, where
\begin{equation}\label{eq:elastic-net-mop2sop}
    \begin{split}
        \mu(w)     &= \frac{w_2}{w_1},\\
        \lambda(w) &= \frac{w_3 + \varepsilon}{w_1},
    \end{split}
\end{equation}
and $\varepsilon$ is given in \cref{eqn:elastic-net-mop}.
In particular, the image of the mapping $(\theta^*, \tilde f_I \circ \theta^*)|_{\Delta^2_I}: \Delta^2_I \to G^*(\tilde f_I)$ for $I=\set{1}, \set{1,2}, \set{1,3}$ approximates\footnote{This image does not exactly coincide with the original solutions due to $\varepsilon$ introduced in \cref{eqn:elastic-net-mop}.} to an OLS solution, the lasso solution path, and the ridge solution path, respectively.
The equations in \cref{eq:elastic-net-mop2sop} are easily obtained by comparing \cref{eqn:elastic-net,eqn:elastic-net-mop}.

\begin{example}\label{thm:example}
    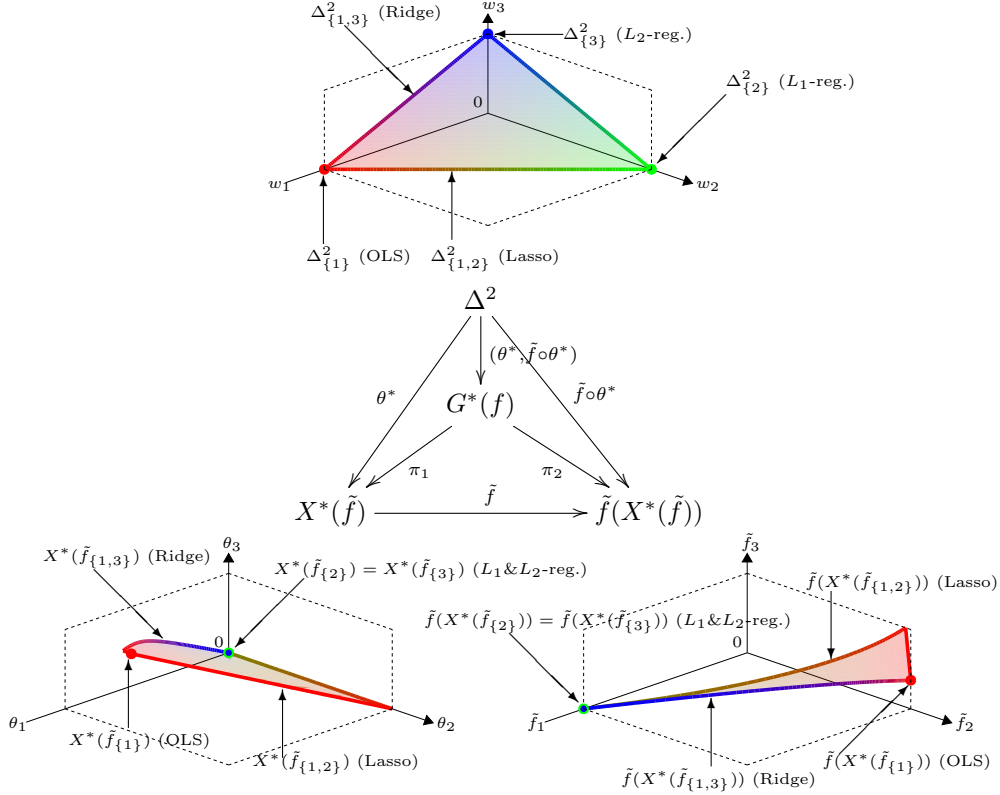
\begin{figure}[tpb]
        \begin{overpic}[width=.45\textwidth]{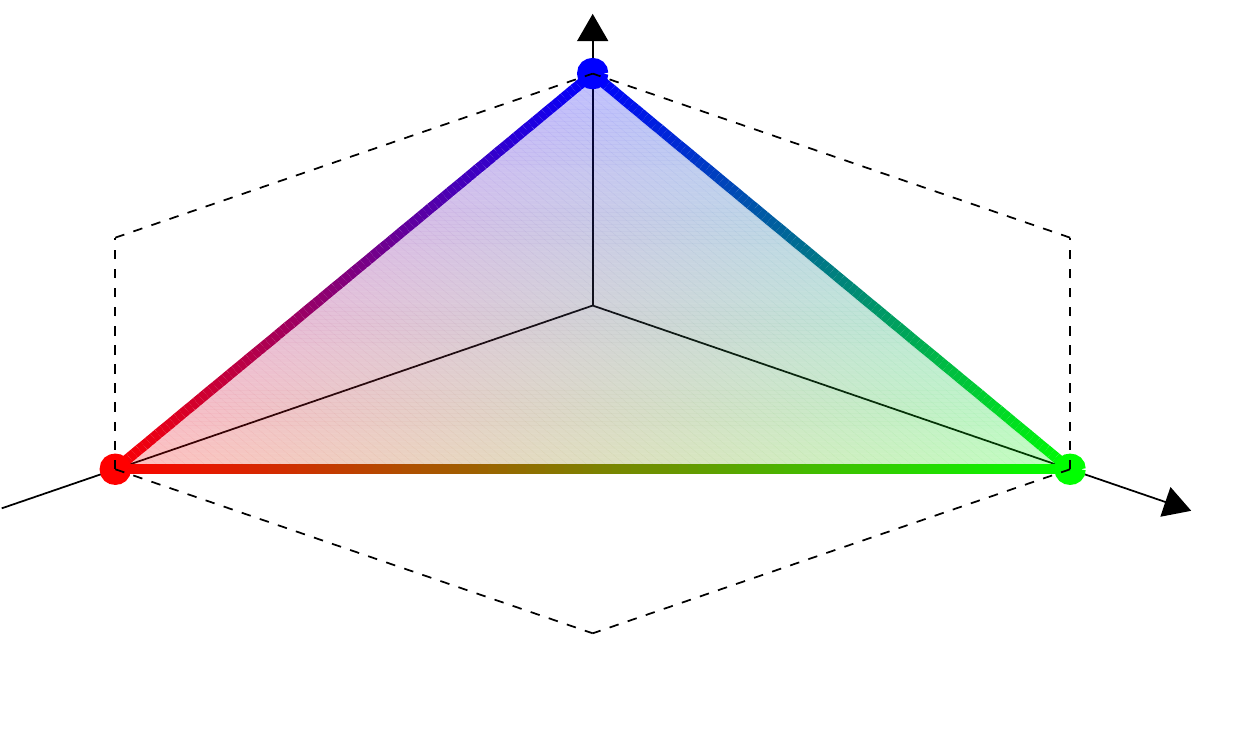}
            \put( 44, 37){\tiny $0$}
            \put( -4, 18){\tiny $w_1$}
            \put( 96, 18){\tiny $w_2$}
            \put( 46, 60){\tiny $w_3$}
            \put(  5,  1){\tiny $\Delta^2_{\{1\}}$ (OLS)}
            \put(  9,  5){\vector(0, 1){16}}
            \put(103, 42){\tiny $\Delta^2_{\{2\}}$ ($L_1$-reg.)}
            \put(103, 40){\vector(-1,-1){16}}
            \put( 65, 53){\tiny $\Delta^2_{\{3\}}$ ($L_2$-reg.)}
            \put( 65, 54){\vector(-1, 0){16}}
            \put( 34,  1){\tiny $\Delta^2_{\{1,2\}}$ (Lasso)}
            \put( 39,  6){\vector(0, 1){16}}
            \put(  6, 58){\tiny $\Delta^2_{\{1,3\}}$ (Ridge)}
            \put( 14, 56){\vector(1,-1){16}}
        \end{overpic}
    \[\xymatrix{
             & \Delta^2 \ar[ldd]_{\theta^*} \ar[d]^{(\theta^*, \tilde f \circ \theta^*)} \ar[rdd]^{\tilde f \circ \theta^*} & \\
             & G^*(f) \ar[ld]^{\pi_1} \ar[rd]_{\pi_2}& \\
              X^*(\tilde f) \ar[rr]^{\tilde f} & & \tilde f(X^*(\tilde f))
        }\]
        \begin{overpic}[width=.45\textwidth]{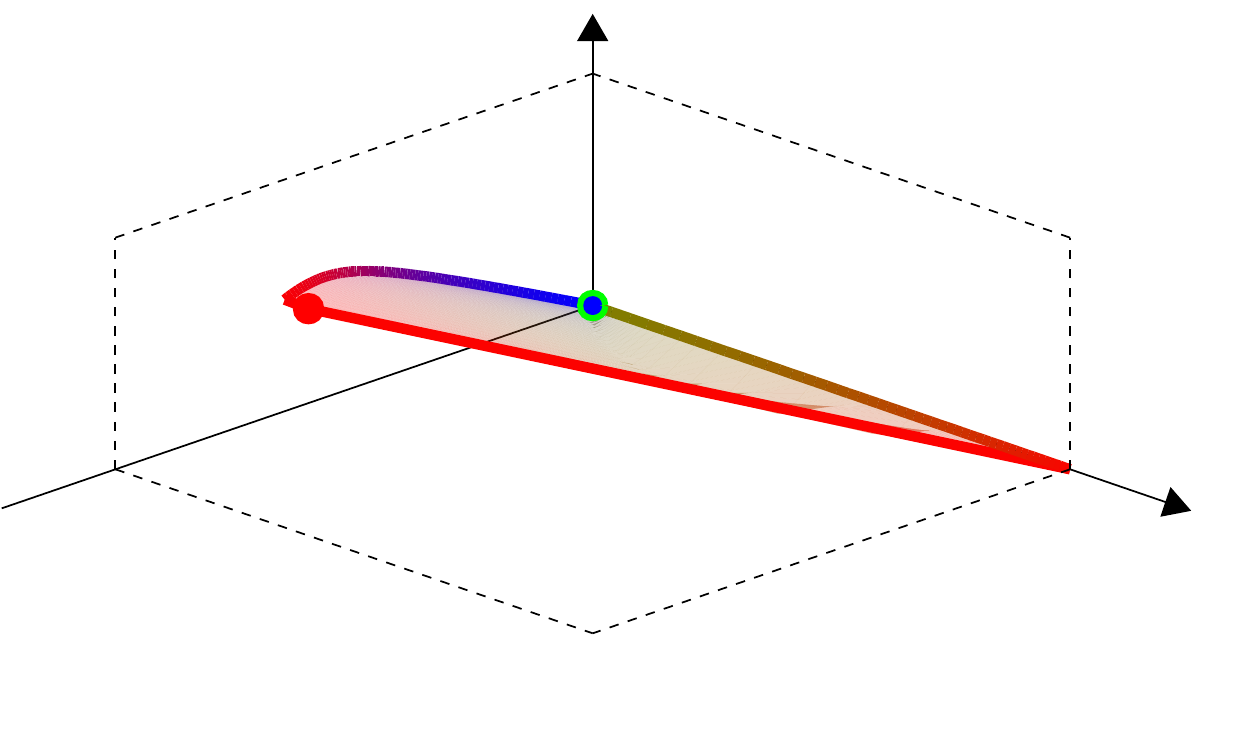}
            \put( 44, 37){\tiny $0$}
            \put( -4, 18){\tiny $\theta_1$}
            \put( 96, 18){\tiny $\theta_2$}
            \put( 46, 60){\tiny $\theta_3$}
            \put( 10, 14){\tiny $X^*(\tilde f_{\{1\}})$ (OLS)}
            \put( 24, 18){\vector(0, 1){16}}
            \put( 58, 54){\tiny $X^*(\tilde f_{\{2\}})=X^*(\tilde f_{\{3\}})$ ($L_1$\&$L_2$-reg.)}
            \put( 65, 53){\vector(-1,-1){16}}
            \put( 53,  9){\tiny $X^*(\tilde f_{\{1,2\}})$ (Lasso)}
            \put( 60, 11){\vector(0, 1){16}}
            \put(  4, 57){\tiny $X^*(\tilde f_{\{1,3\}})$ (Ridge)}
            \put( 16, 55){\vector(1,-1){16}}
        \end{overpic}\hspace{12mm}%
        \begin{overpic}[width=.45\textwidth]{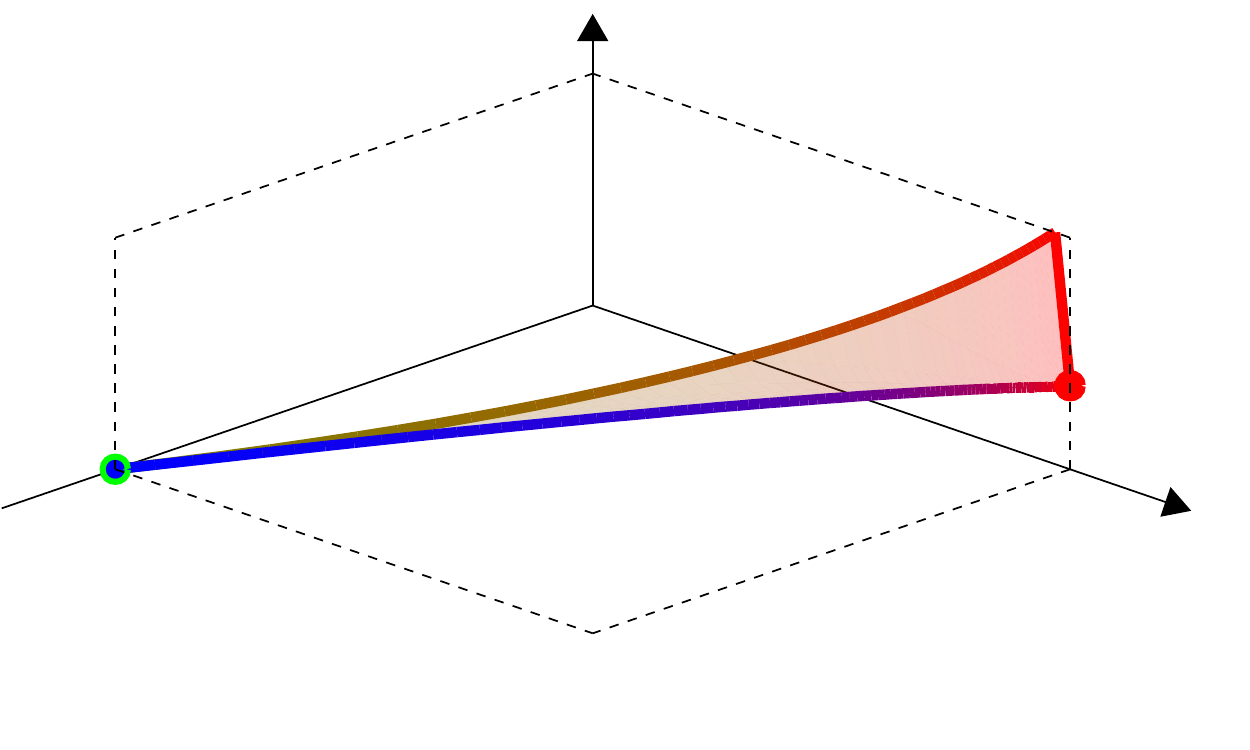}
            \put( 44, 37){\tiny $0$}
            \put( -4, 18){\tiny $\tilde f_1$}
            \put( 96, 18){\tiny $\tilde f_2$}
            \put( 46, 60){\tiny $\tilde f_3$}
            \put( 65,  9){\tiny $\tilde f(X^*(\tilde f_{\{1\}}))$ (OLS)}
            \put( 69, 12){\vector(1, 1){16}}
            \put(-28, 42){\tiny $\tilde f(X^*(\tilde f_{\{2\}}))=\tilde f(X^*(\tilde f_{\{3\}}))$ ($L_1$\&$L_2$-reg.)}
            \put( -8, 40){\vector(1,-1){16}}
            \put( 61, 51){\tiny $\tilde f(X^*(\tilde f_{\{1,2\}}))$ (Lasso)}
            \put( 67, 50){\vector(0,-1){16}}
            \put( 18,  5){\tiny $\tilde f(X^*(\tilde f_{\{1,3\}}))$ (Ridge)}
            \put( 39,  9){\vector(0, 1){16}}
        \end{overpic}\\
        \caption{Solution mapping of the elastic net in \cref{thm:example}.}\label{fig:solution-mapping}
    \end{figure}
    \Cref{fig:solution-mapping} shows the solution mapping of the elastic net for the following data:
    \[
        X=\begin{pmatrix}
             1 &  2 &  3\\
             6 &  5 &  4\\
             7 &  8 &  9\\
            12 & 11 & 10
        \end{pmatrix},\quad
        y=\begin{pmatrix}
            1\\
            2\\
            3\\
            4
        \end{pmatrix}.
    \]
    By $w=(1,0,0)$, an OLS solution (the red point in the figure) is obtained.
    With $w=(t,1-t,0)$ for $0<t<1$, the lasso solution path (the red-green curve in the figure) is obtained.
    With $w=(t,0,1-t)$ for $0<t<1$, the ridge solution path (the red-blue curve in the figure) is obtained.
\end{example}

\subsection{B\'ezier simplex fitting}\label{sec:bezier-simplex}
Let $\N$ be the set of nonnegative integers and
\[
    \N^m_d =\Set{(i_1,\dots,i_m)\in\N^m | \sum_{k=1}^m i_k = d}.
\]
An $(m-1)$-dimensional B\'ezier simplex of degree $d$ in $\R^n$ is a mapping $b:\Delta^{m-1}\to\R^n$ specified by control points $p_i \in \R^n$ $(i\in\N^m_d)$:
\[
    b(w) = \sum_{i\in\mathbb N^m_d} \binom{d}{i} w^i p_i,
\]
where $\binom{d}{i}=d!/(i_1! i_2! \cdots i_m!)$ and $w^i = w_1^{i_1} w_2^{i_2} \cdots w_m^{i_m}$.

It is known that any continuous mapping from a simplex to a Euclidean space can be approximated by a B\'ezier simplex:
\begin{theorem}[{\cite[Theorem~1]{Kobayashi2019}}]\label{thm:approx}
    Let $\phi: \Delta^{m-1}\to\R^n$ be a continuous mapping.
    There exists an infinite sequence of B\'ezier simplices $b^i:\Delta^{m-1}\to\R^n$ such that
    \[
        \lim_{i\to\infty}\sup_{w\in\Delta^{m-1}}\norm{\phi(w)-b^i(w)} = 0.
    \]
\end{theorem}

Tanaka et al.~\cite{Tanaka2020} proposed an algorithm for adjusting control points so that a B\'ezier simplex of fixed degree fits a sample of the graph of a continuous mapping $\phi: \Delta^{m-1}\to\R^n$.
Based on the above theorem and algorithm, one can build a B\'ezier simplex approximation of the solution mapping $(\theta^*, \tilde f\circ\theta^*): \Delta^2\to\R^{n+3}$ in \cref{eq:solution-mapping} by using a sample of it, where the input is a hyper-parameter vector $w$, and the output is a concatenation of a model parameter vector $\theta^*(w)$ and a loss vector $\tilde f(\theta^*(w))$.

\subsection{Experiments}\label{sec:experiments}
We empirically confirmed the aforementioned theory that the solution mapping of the multi-objective elastic net can be approximated by a B\'ezier simplex.
We used eight datasets shown in \cref{tab:datasets}, all of which are adopted from UCI Machine Learning Repository \cite{Dua2019}.
To train the elastic net for each dataset, the attributes were split into predictors and responses according to the description on the dataset's web page.
If the dataset contains two responses, the elastic net prepares two sets of all predictors, each of which predicts each response, i.e., the elastic net has 206 predictors for Residential Building and 770 predictors for Slice Localization.
The predictors and responses were normalized as the range of each attribute becomes the unit interval.

For each dataset, a sample of the solution mapping was created as follows.
We generated 5151 hyper-parameters as grid points on $\Delta^2$:
\[
    w = \frac{1}{100}(n_1, n_2, n_3) \text{ such that } n_1, n_2, n_3 \in \set{0, 1, \dots, 100},\ n_1 + n_2 + n_3 = 100.
\]
For each point $w$, we computed the value of $\theta^*(w)$ and $\tilde f \circ \theta^*(w)$.
To do so, the weight $w = (w_1, w_2, w_3)$ was converted to the regularization coefficient $(\mu, \lambda)$ according to \cref{eq:elastic-net-mop2sop}, where the magnitude of perturbation $\varepsilon$ was set to 1E-16.
Then, the original elastic net problem \cref{eqn:elastic-net} was solved by the coordinate descent method.

\begin{center}
\begin{threeparttable}[t]
    \caption{Datasets}
    \label{tab:datasets}
    \begin{tabular}{lrrr}
    \toprule
                                                     & \multicolumn{2}{c}{Attributes}\\
                                 \cmidrule(lr){2-3}
    Dataset                                          & Predictors & Responses & Instances \\
    \midrule
    Blog Feedback\tnote{a}~\cite{Buza2014}           &   280 & 1 & 60,021 \\
    Fertility\tnote{b}~\cite{Gil2012}                &     9 & 1 &    100 \\
    Forest Fires\tnote{c}~\cite{Cortez2007}          &    12 & 1 &    517 \\
    QSAR Fish Toxicity\tnote{d}~\cite{Cassotti2015}  &     6 & 1 &    908 \\
    Residential Building\tnote{e}~\cite{Rafiei2016}  &   103 & 2 &    372 \\
    Slice Localization\tnote{f}~\cite{Graf2011}      &   385 & 2 & 53,500 \\
    Wine\tnote{g}~\cite{Aeberhard1992}               &    11 & 1 &    178 \\
    Yacht Hydrodynamics\tnote{h}~\cite{Ortigosa2007} &     6 & 1 &    308 \\
    \bottomrule
    \end{tabular}
    {\tiny
    \begin{tablenotes}
        \item[a]{\url{https://archive.ics.uci.edu/ml/datasets/BlogFeedback}}
        \item[b]{\url{https://archive.ics.uci.edu/ml/datasets/Fertility}}
        \item[c]{\url{https://archive.ics.uci.edu/ml/datasets/Forest+Fires}}
        \item[d]{\url{https://archive.ics.uci.edu/ml/datasets/QSAR+fish+toxicity}}
        \item[e]{\url{https://archive.ics.uci.edu/ml/datasets/Residential+Building+Data+Set}}
        \item[f]{\url{https://archive.ics.uci.edu/ml/datasets/Relative+location+of+CT+slices+on+axial+axis}}
        \item[g]{\url{https://archive.ics.uci.edu/ml/datasets/wine}}
        \item[h]{\url{https://archive.ics.uci.edu/ml/datasets/Yacht+Hydrodynamics}}
    \end{tablenotes}
    }
\end{threeparttable}
\end{center}
We trained B\'ezier simplices by the all-at-once method \cite{Tanaka2020} using the grid points $w$ as inputs and the elastic net results $(\theta^*(w), \tilde f \circ \theta^*(w))$ as outputs.
For each setting below, 10 trials were run with different random seeds:
\begin{description}
    \item[Train-test split ratio] $\text{Train}:\text{Test} = 51:5100$, $257:4894$, $515:4636$, $2575:2576$, $5100:51$.
    \item[Degree of B\'ezier simplex] $d = 1, 2, \dots, 14, 15, 20, 25, 30$.
\end{description}

We carried out experiments on the ITO supercomputer\footnote{\url{https://www.cc.kyushu-u.ac.jp/scp/eng/system/ITO/01_intro.html}}.
The elastic net and the B\'ezier simplex fitting were implemented by the Python package \texttt{scikit-learn 0.24.2} and \texttt{pytorch-bsf 0.0.1}, respectively.

\subsection{Results}\label{sec:results}
\begin{figure}[tpb]
    \subfloat[$\text{Train}:\text{Test} = 51:5100$\label{fig:qsar-001}]{
        \includegraphics[width=.4\textwidth]{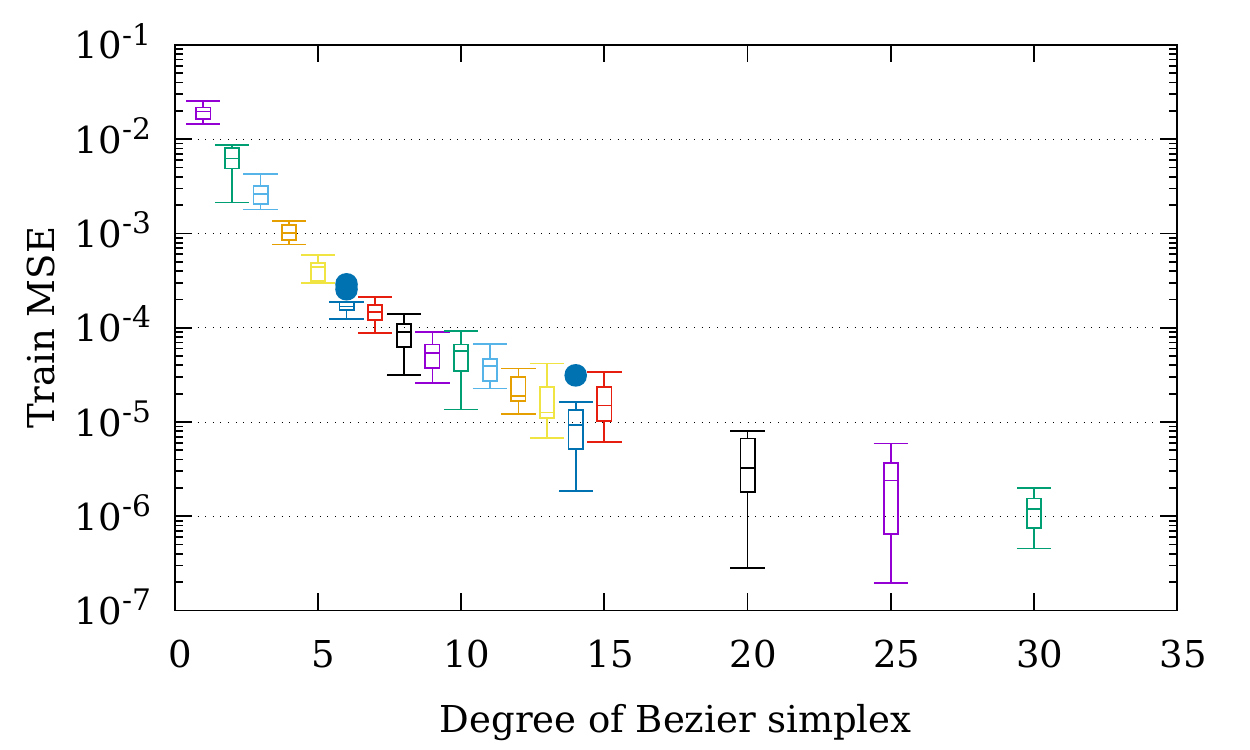}
        \includegraphics[width=.4\textwidth]{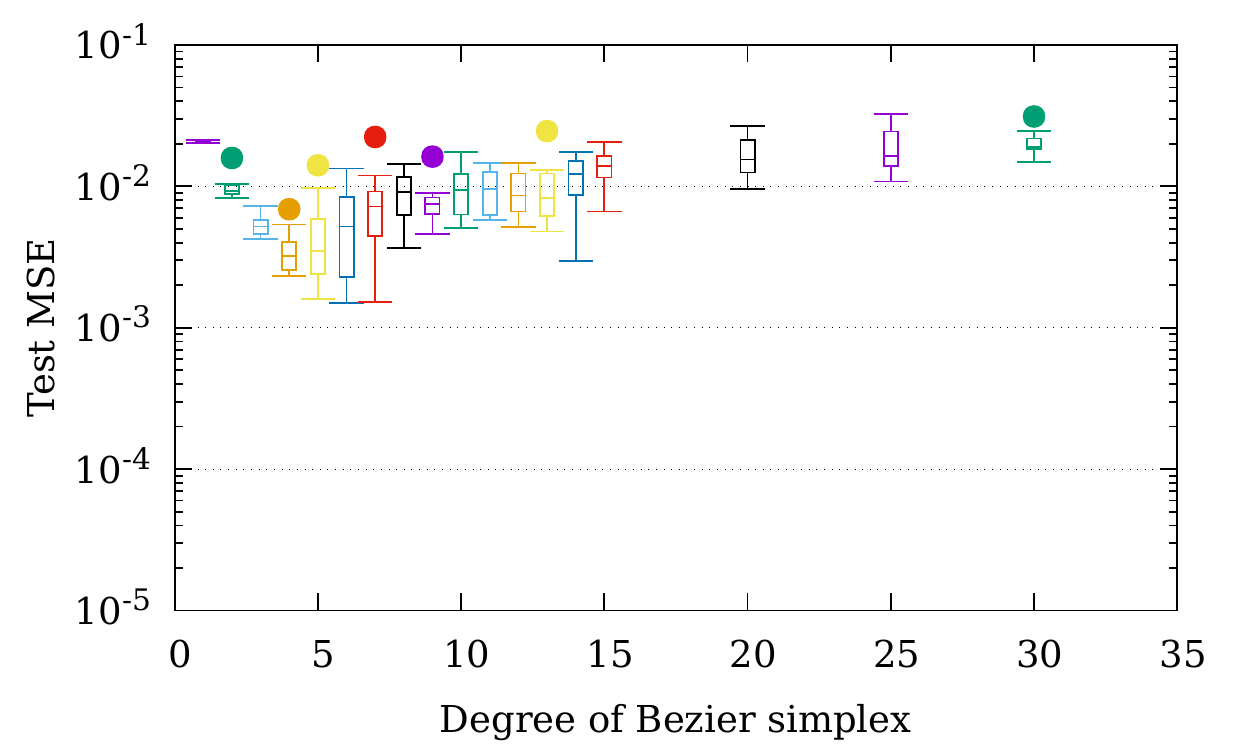}
    }\\
    \subfloat[$\text{Train}:\text{Test} = 257:4894$\label{fig:qsar-005}]{
        \includegraphics[width=.4\textwidth]{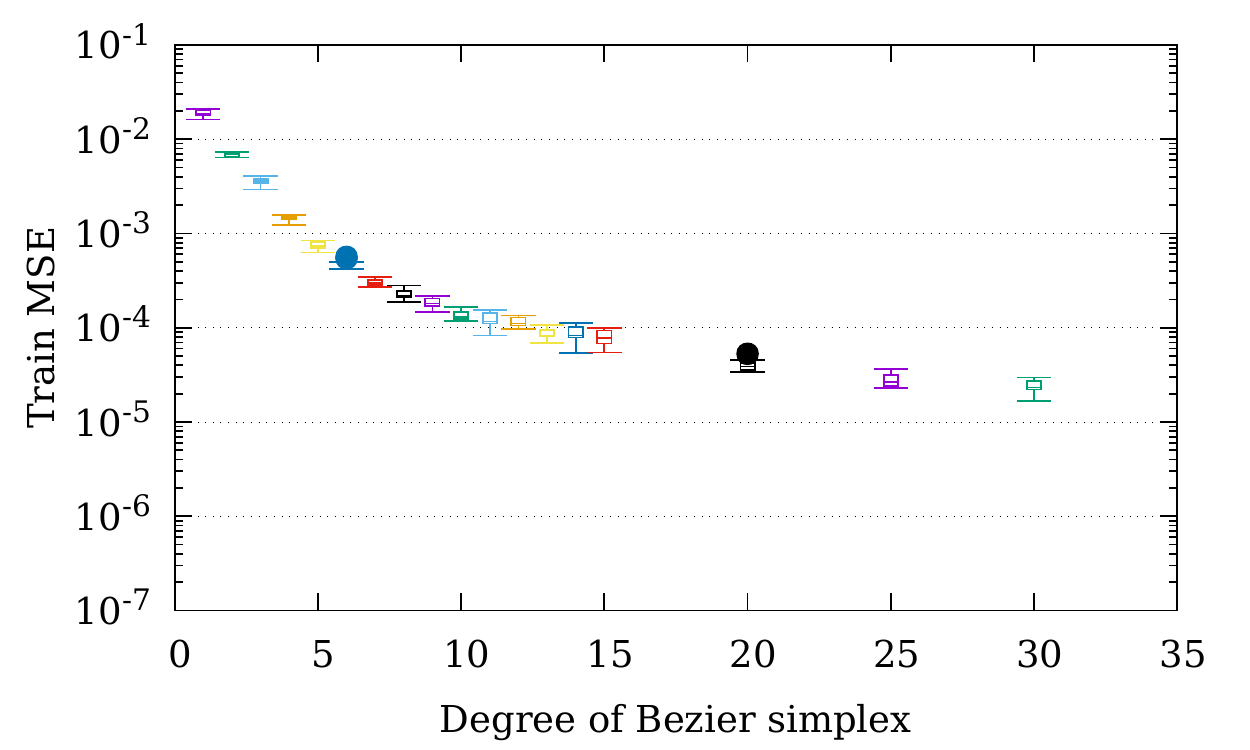}
        \includegraphics[width=.4\textwidth]{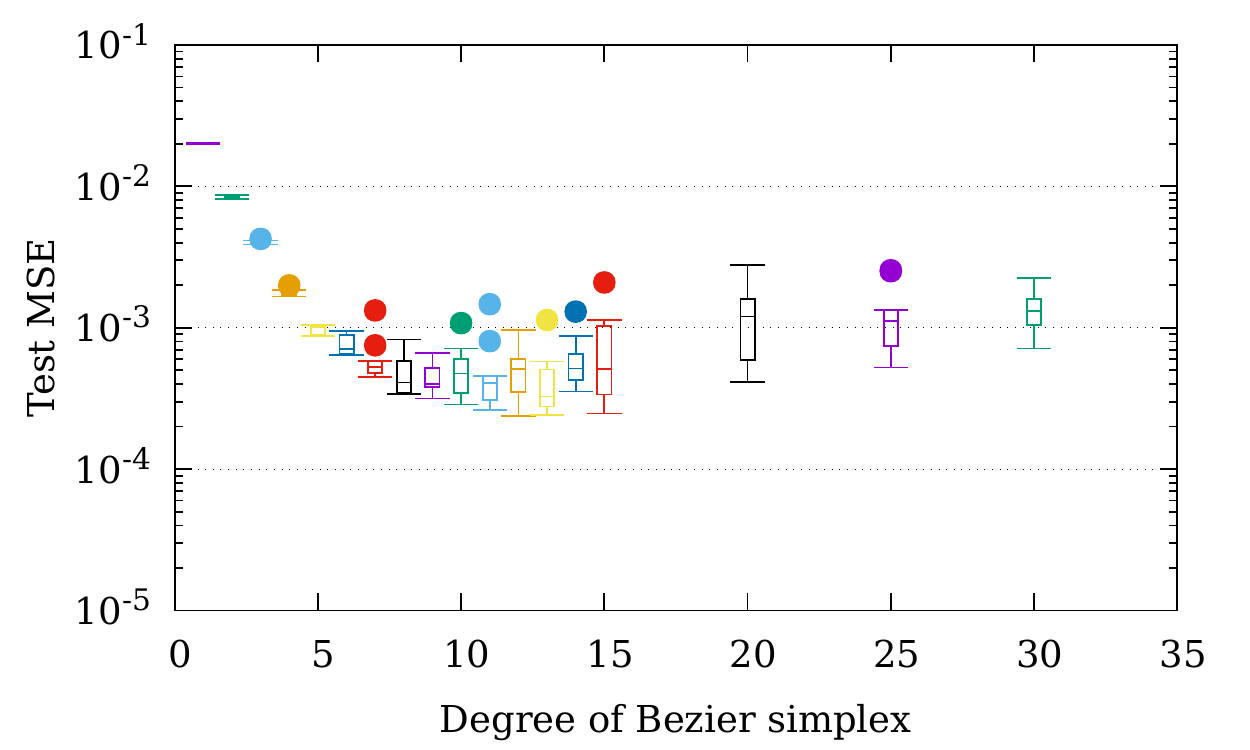}
    }\\
    \subfloat[$\text{Train}:\text{Test} = 515:4636$\label{fig:qsar-010}]{
        \includegraphics[width=.4\textwidth]{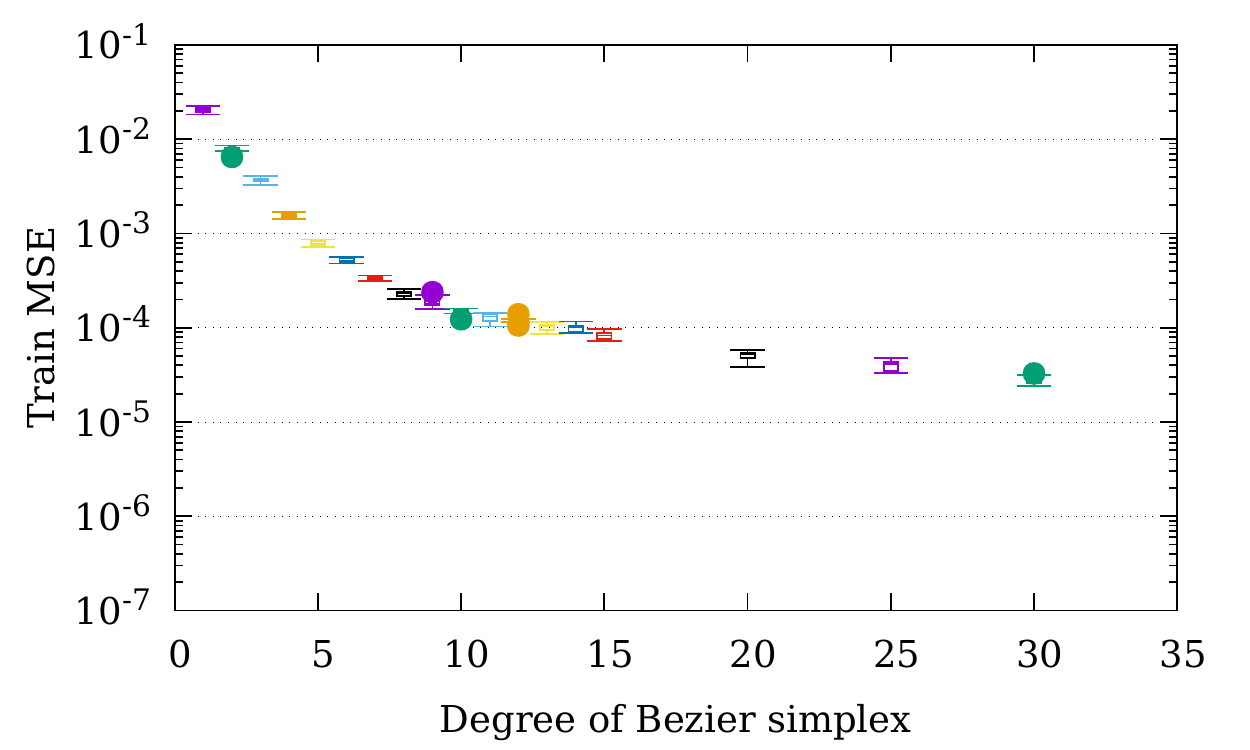}
        \includegraphics[width=.4\textwidth]{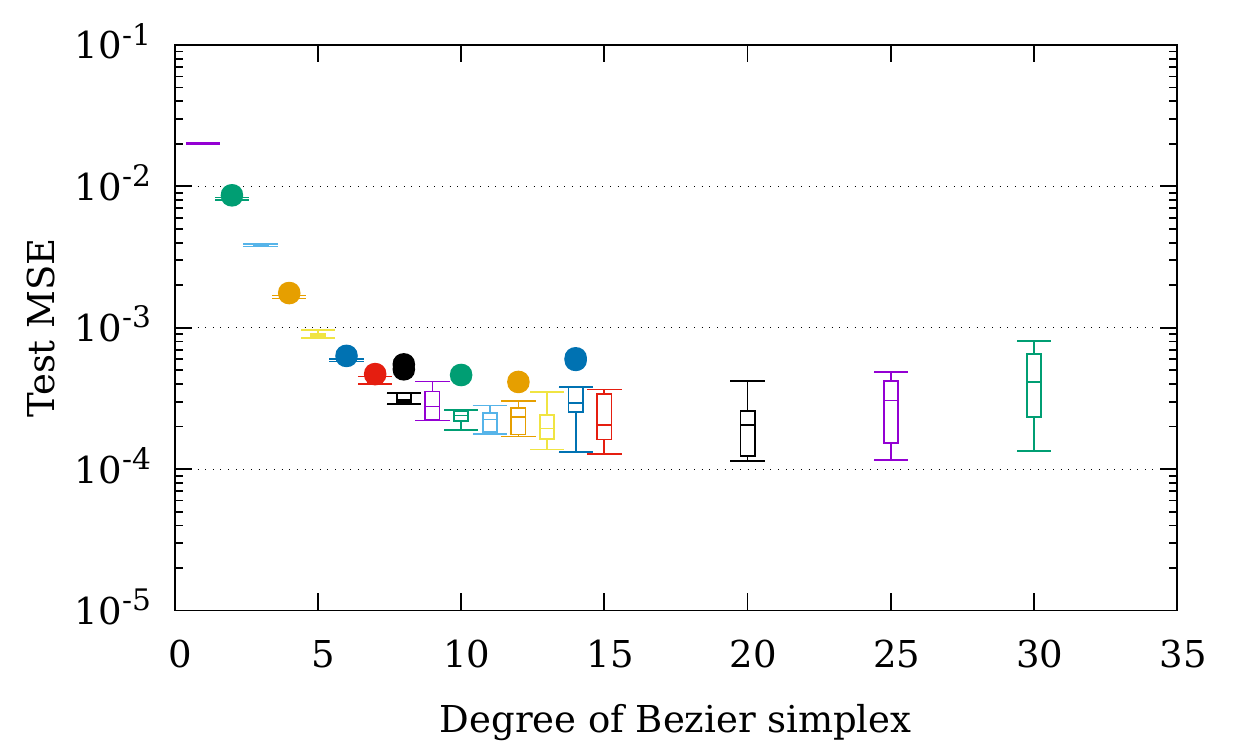}
    }\\
    \subfloat[$\text{Train}:\text{Test} = 2575:2576$\label{fig:qsar-050}]{
        \includegraphics[width=.4\textwidth]{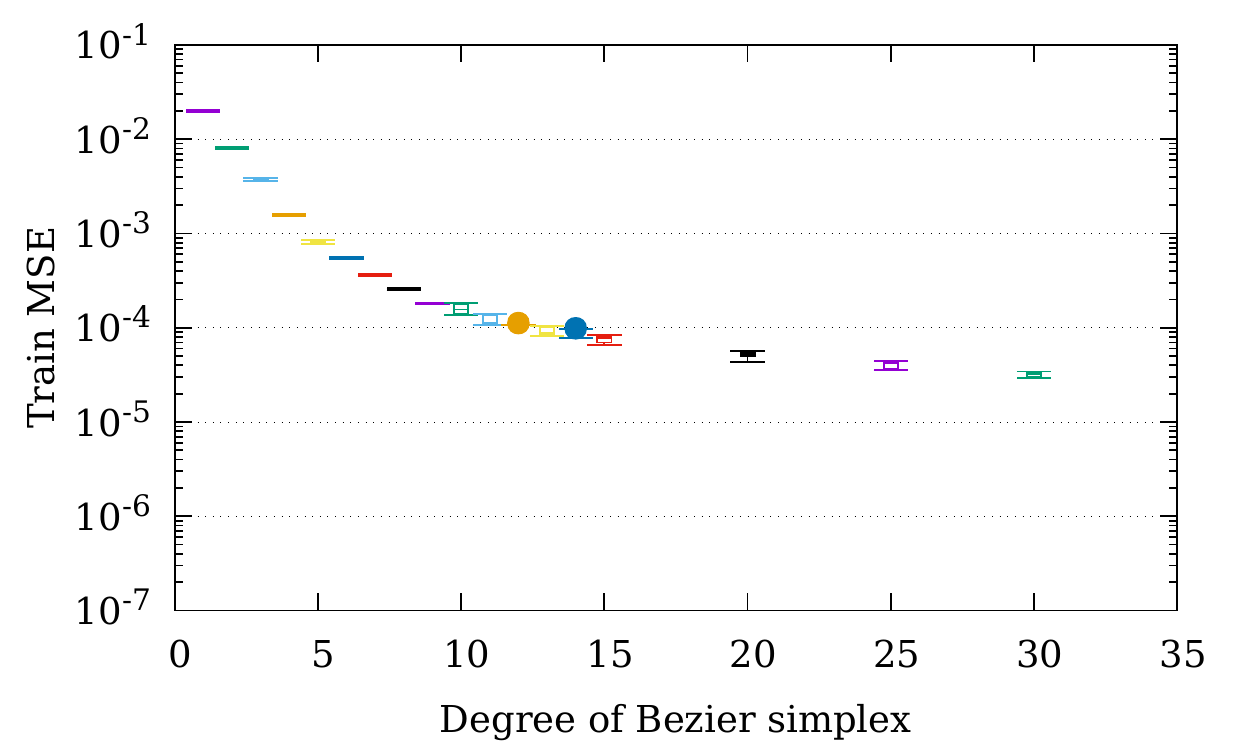}
        \includegraphics[width=.4\textwidth]{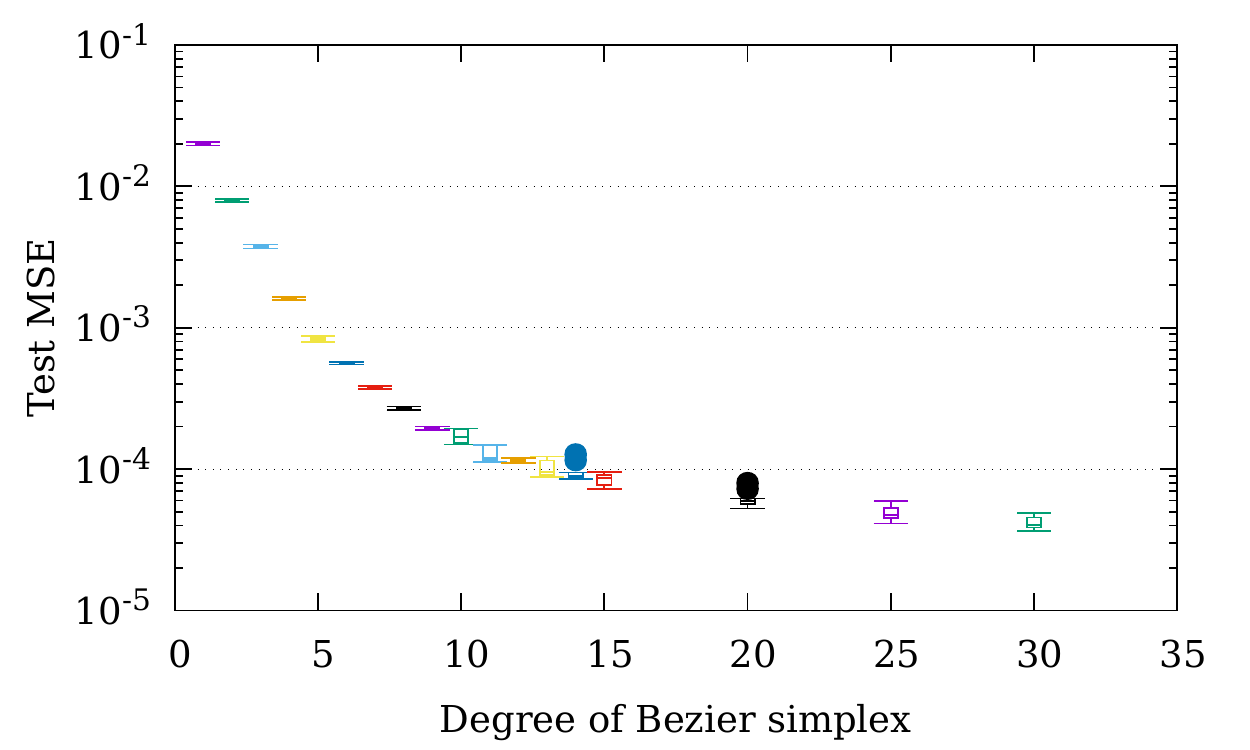}
    }\\
    \subfloat[$\text{Train}:\text{Test} = 5100:51$\label{fig:qsar-099}]{
        \includegraphics[width=.4\textwidth]{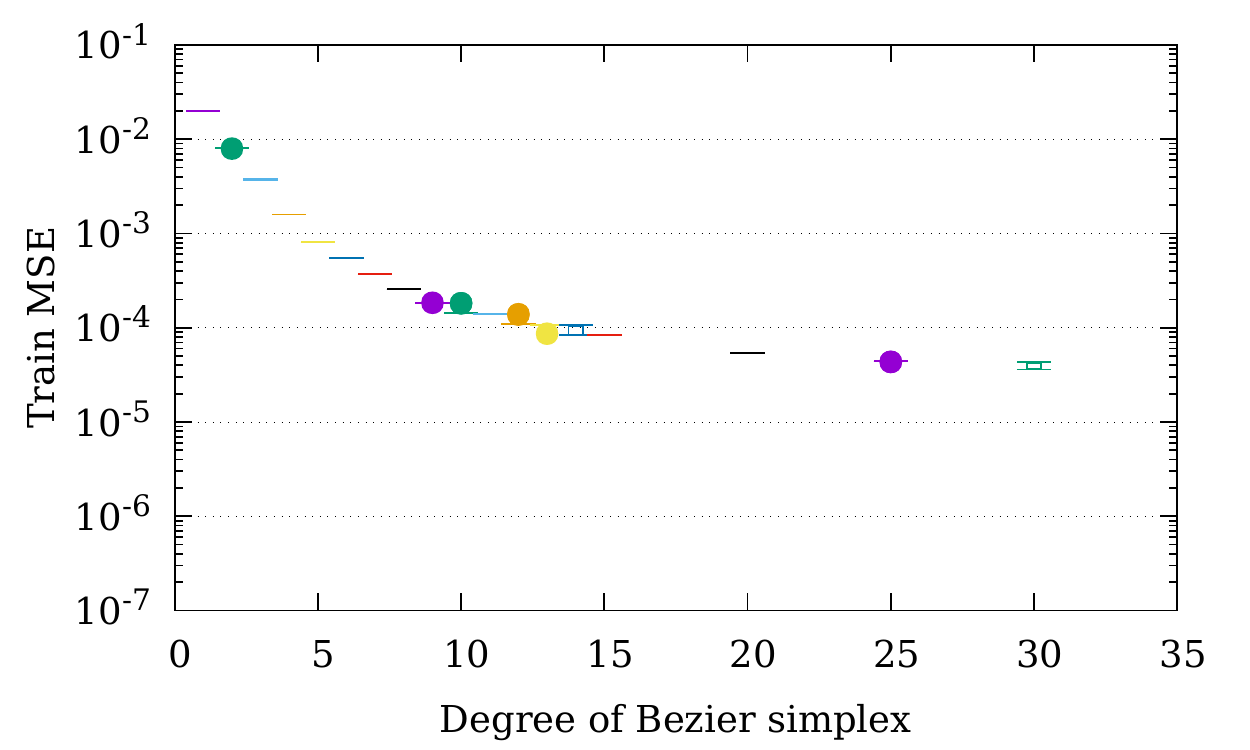}
        \includegraphics[width=.4\textwidth]{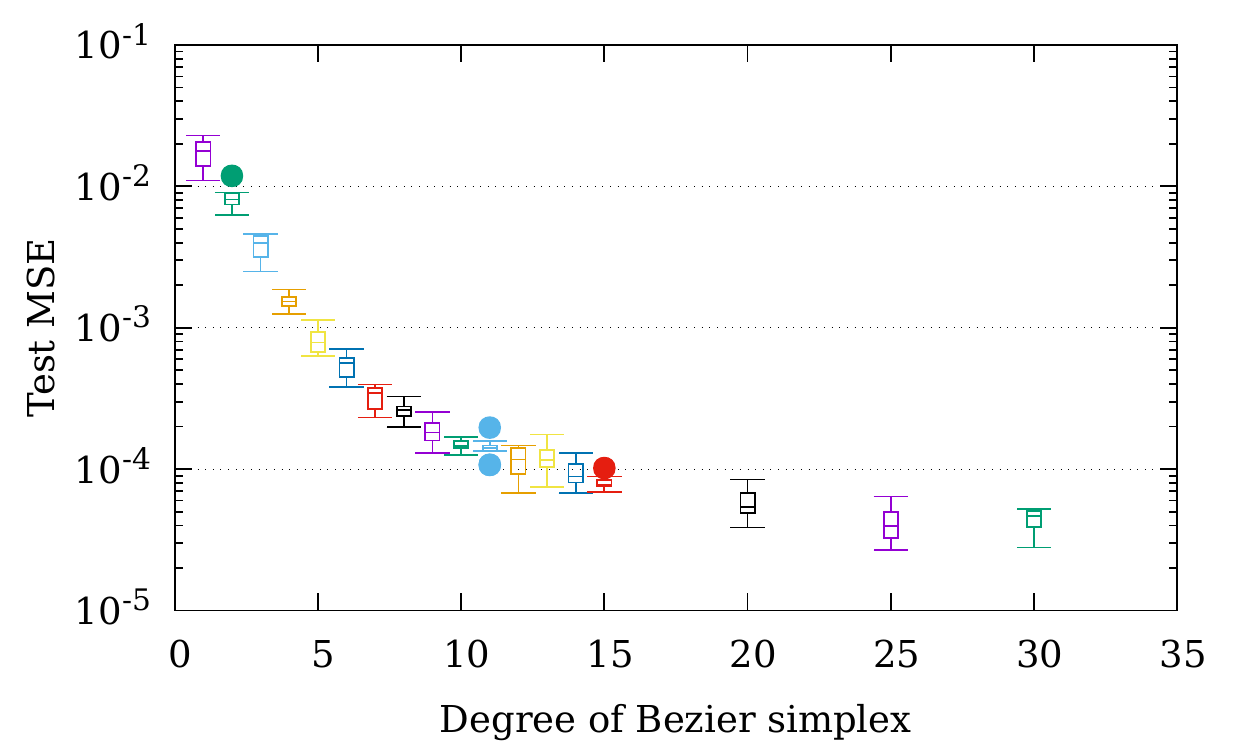}
    }
    \caption{Degree versus MSE on QSAR Fish Toxicity.}\label{fig:qsar-err}
\end{figure}
\Cref{fig:qsar-err} shows how train and test MSEs change as the degree increases in fitting a B\'ezier simplex for various train-test split ratios of QSAR Fish Toxicity dataset.
When the ratio is $51:5100$ (\cref{fig:qsar-001}), degree $d=4$ is optimal to minimize the average test MSE.
In case of the ratio being 257:4894 (\cref{fig:qsar-005}), the optimal degree becomes $d=10$, and its average test MSE is lower than that of the ratio being $51:5100$ and $d=4$.
As the ratio goes higher (\cref{fig:qsar-010,fig:qsar-050,fig:qsar-099}), the optimal degree keeps increasing, and its average test MSE consistently decreases.
This result agrees with the approximation theorem of B\'ezier simplices (\cref{thm:approx}).

\begin{figure}[tpb]
    \subfloat[Ground truth (5151 elastic net models trained with varying hyper-parameters).\label{fig:scatter-gt}]{
        \includegraphics[width=.33\textwidth]{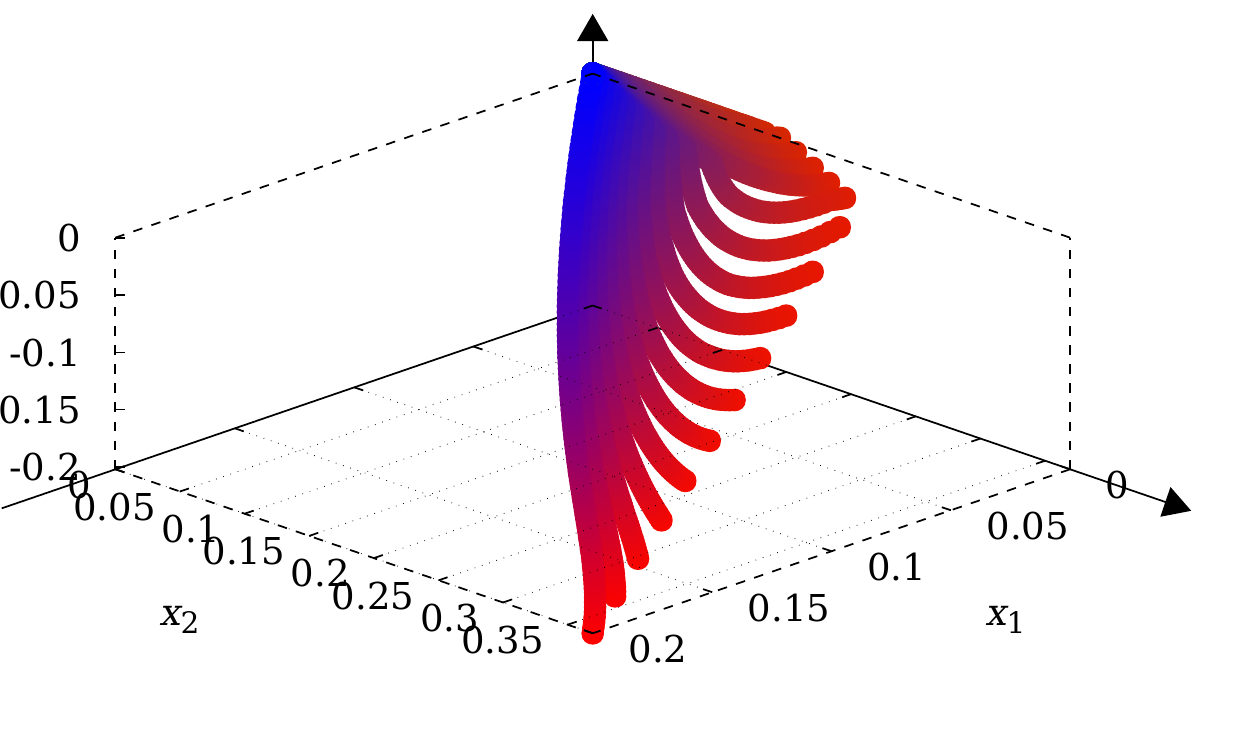}
        \includegraphics[width=.33\textwidth]{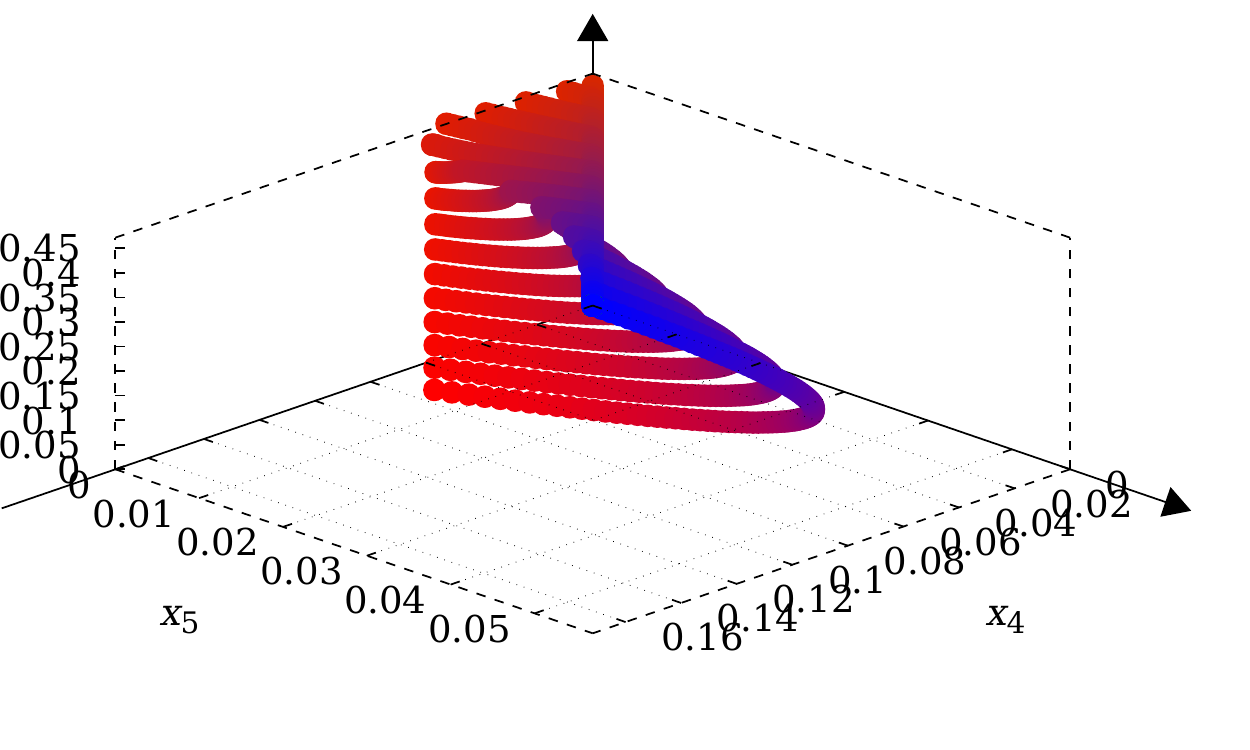}
        \includegraphics[width=.33\textwidth]{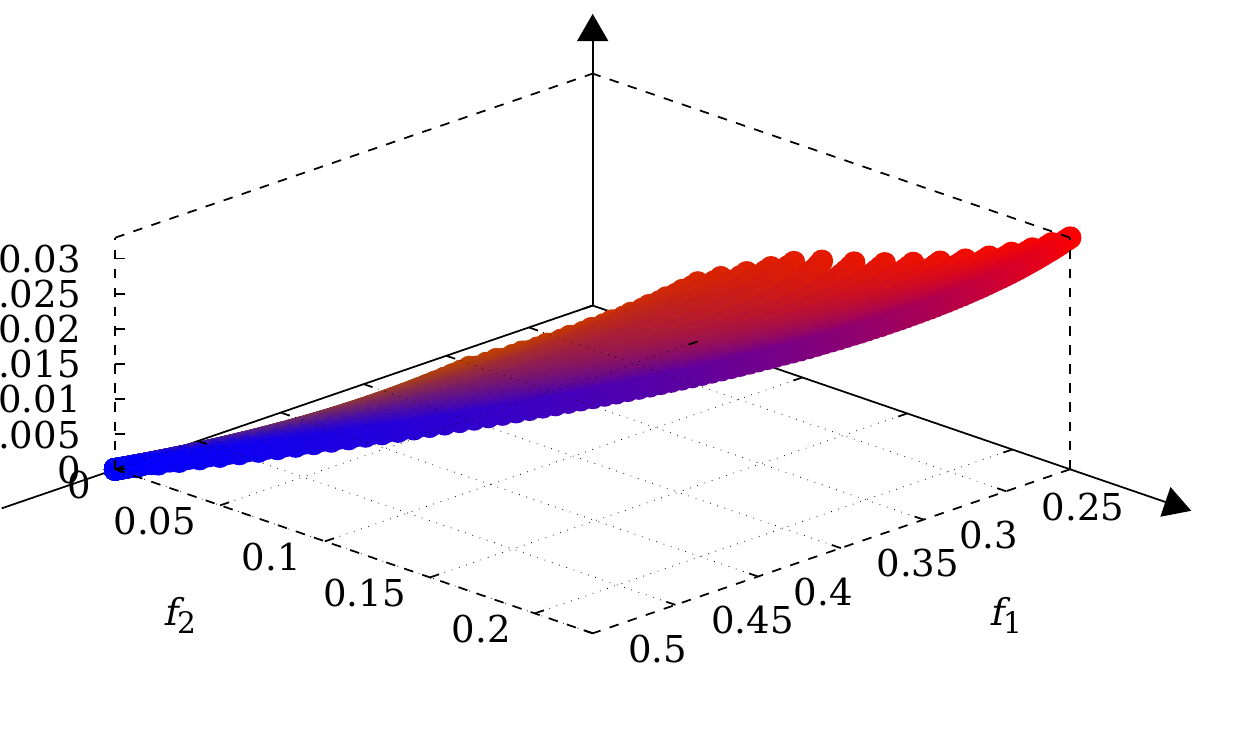}
    }\\
    \subfloat[Large sample approximation (A B\'ezier simplex of $d=25$ trained with 5100 data points).\label{fig:scatter-ls}]{
        \includegraphics[width=.33\textwidth]{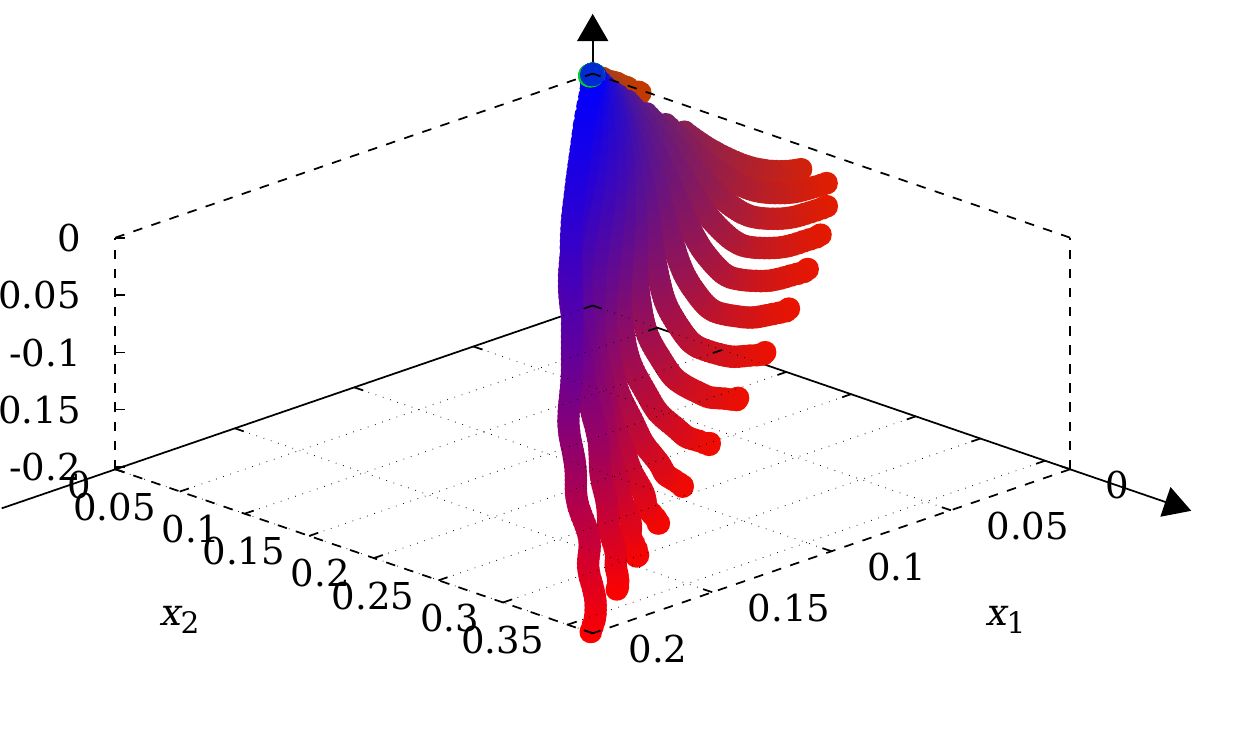}
        \includegraphics[width=.33\textwidth]{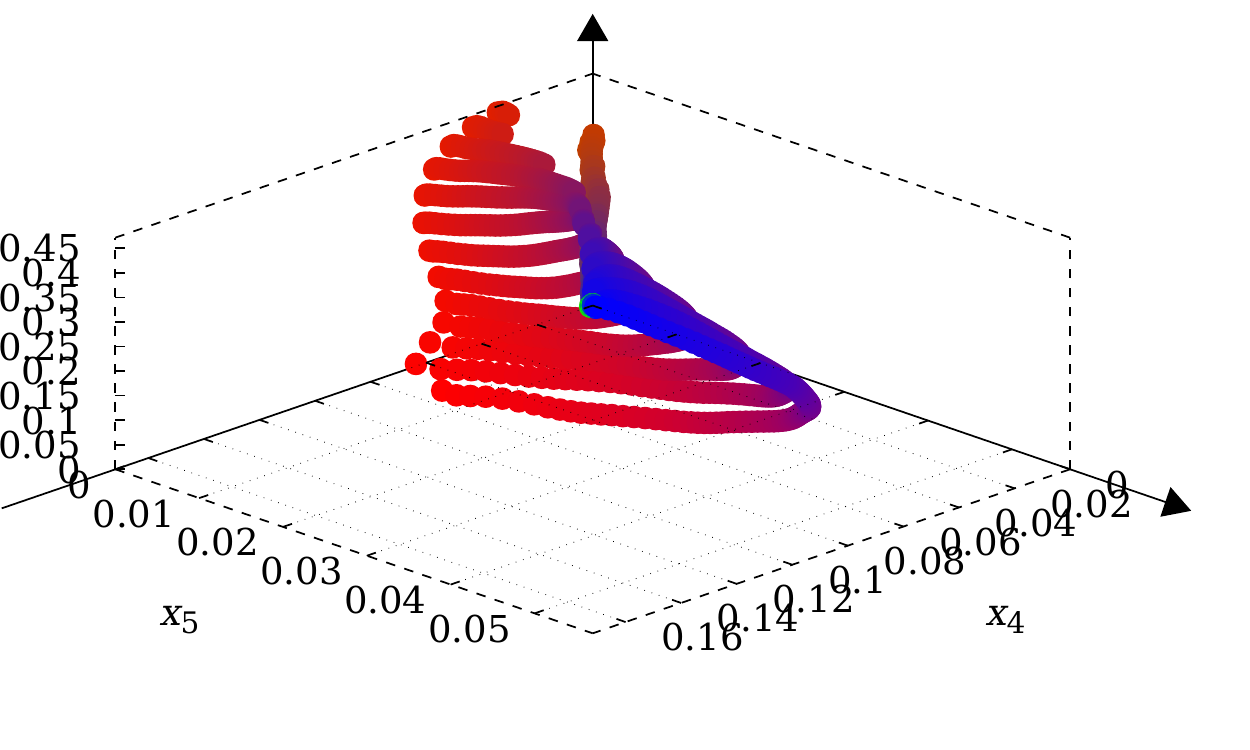}
        \includegraphics[width=.33\textwidth]{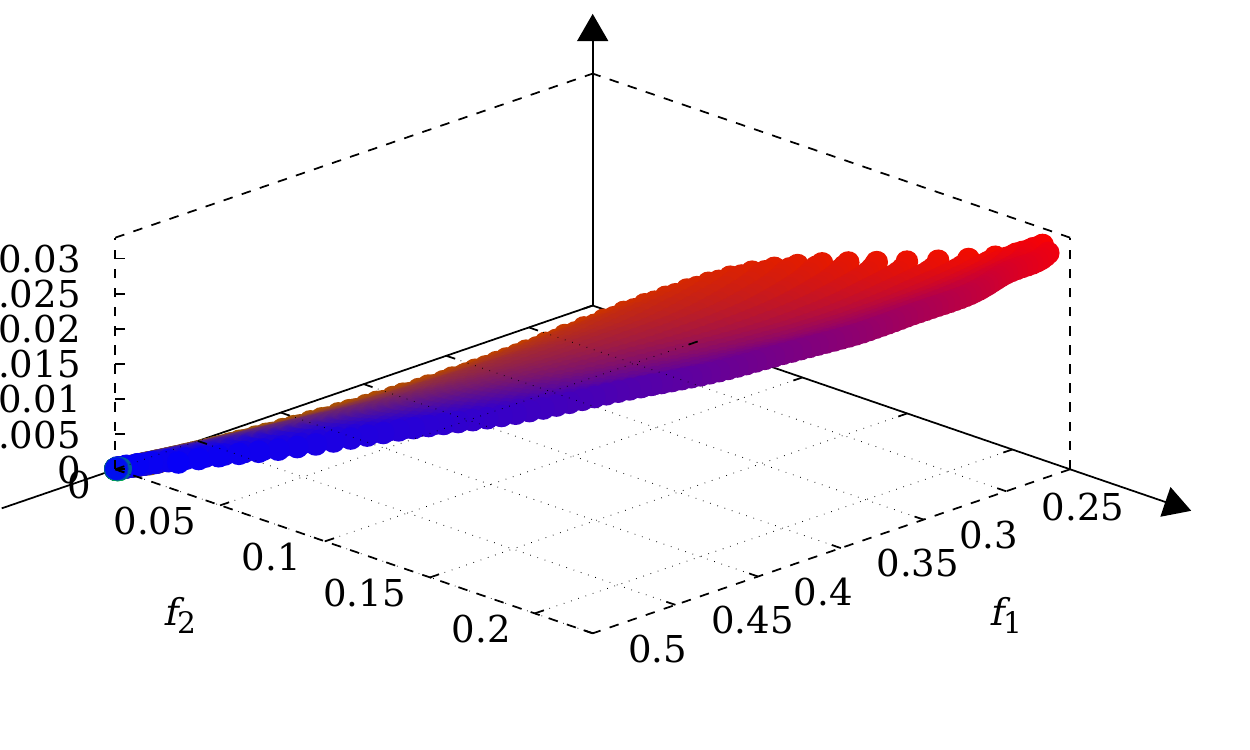}
    }\\
    \subfloat[Small sample approximation (A B\'ezier simplex of $d=4$ trained with 51 data points).\label{fig:scatter-ss}]{
        \includegraphics[width=.33\textwidth]{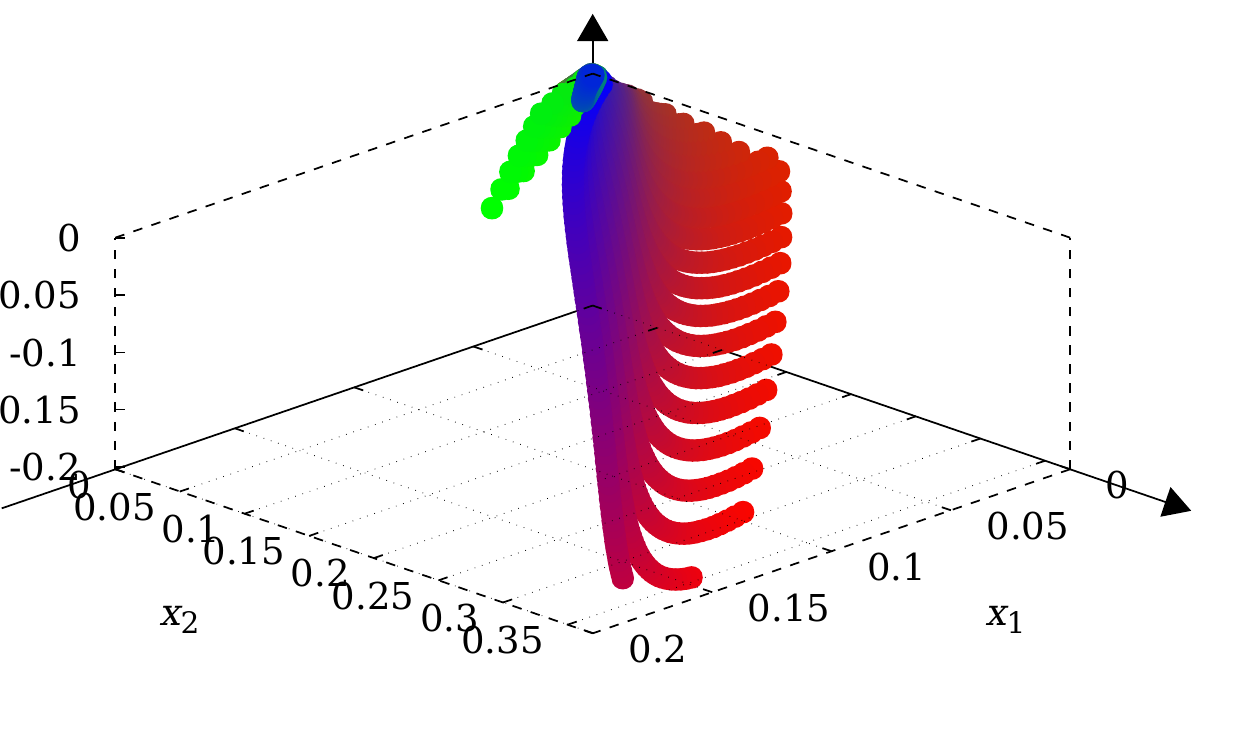}
        \includegraphics[width=.33\textwidth]{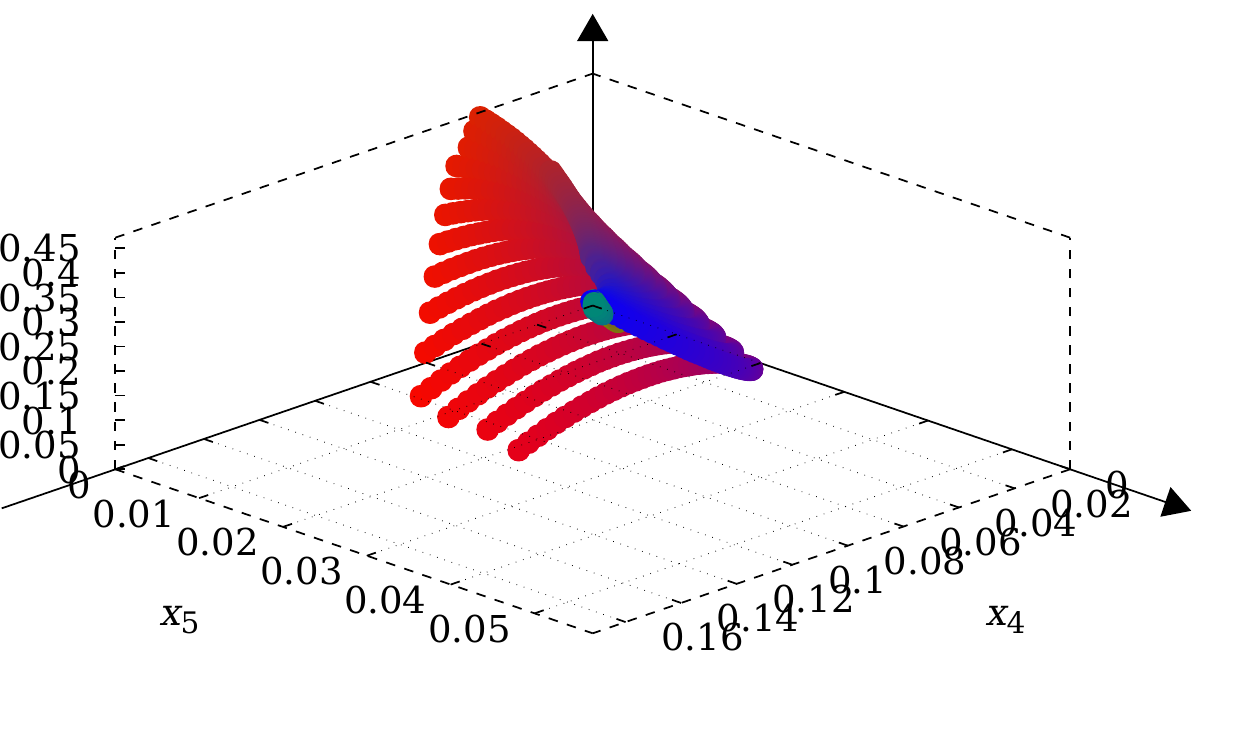}
        \includegraphics[width=.33\textwidth]{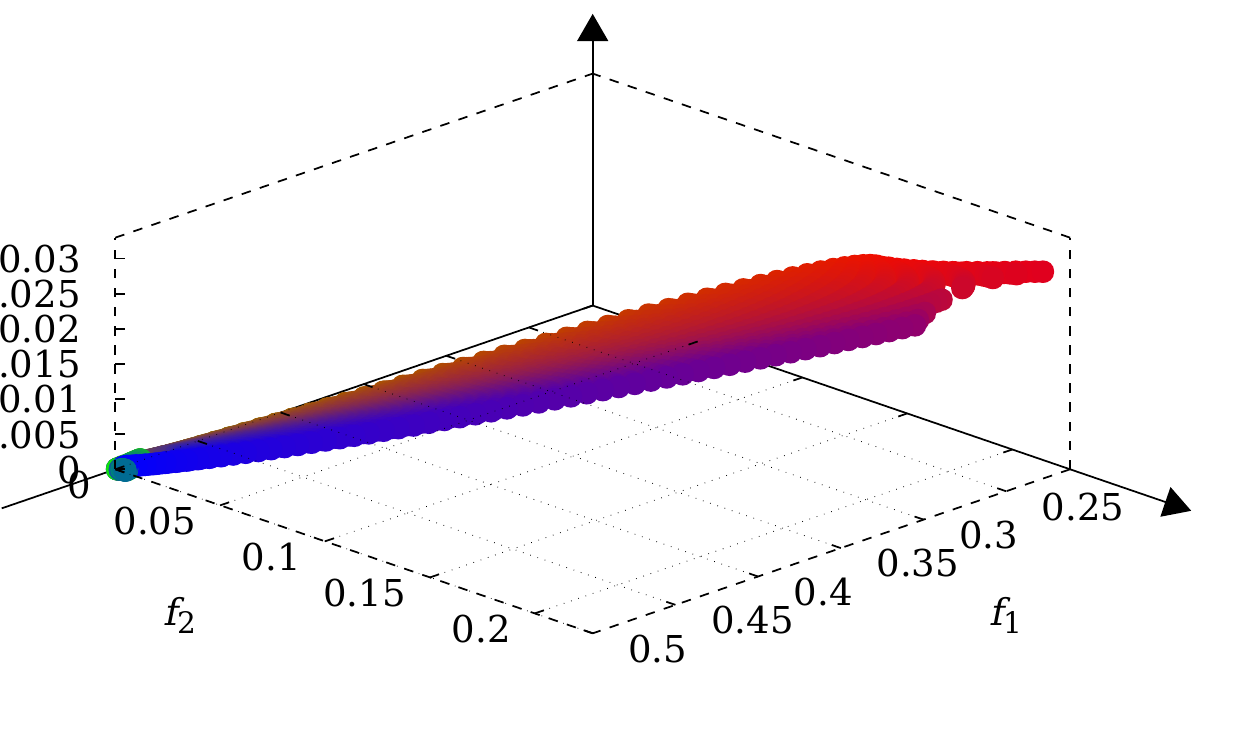}
    }
    \caption{QSAR Fish Toxicity (left: $(\theta^*_1,\theta^*_2,\theta^*_3)(W)$, mid: $(\theta^*_4,\theta^*_5,\theta^*_6)(W)$, right: $\tilde f \circ \theta^*(W)$).}\label{fig:scatter}
\end{figure}
Next, let us visually check the quality of approximation.
\Cref{fig:scatter} compares the ground truth and two approximations, one built from a large sample (train-test split is $5100:51$, and degree is $d=25$) and the other made from a small sample (train-test split is $51:5100$, and degree is $d=4$).
In the figure, each point is colored with its corresponding weight $w$ by converting $(w_1, w_2, w_3)$-coordinates to RGB values.
The color changes continuously on each surface in the ground truth  (\cref{fig:scatter-gt}), confirming \cref{thm:contisurj} that states the mapping $\theta^*$ is continuous (and thus so is $\tilde f \circ \theta^*$).
The large sample approximation (\Cref{fig:scatter-ls}) is close to the ground truth, which means all the elastic net models trained with different hyper-parameters are represented by a B\'ezier simplex as well as their loss values.
It is surprising that even with a small sample (\cref{fig:scatter-ss}), the Pareto front is still well-approximated, which enables us to conduct trade-off analysis on the approximate Pareto front for exploring the preferred hyper-parameter.

\begin{table}[t]
    \centering
    \caption{Optimal degree $d^*$ and its approximation error (average $\pm$ standard deviation over 10 trials).}
    \label{tab:errors}
    \begin{tabular}{ccccc}
    \toprule
                         & \multicolumn{2}{c}{Large sample} & \multicolumn{2}{c}{Small sample} \\ \cmidrule(lr){2-3} \cmidrule(lr){4-5}
    Dataset              & $d^*$ & Test MSE & $d^*$ & Test MSE\\
    \midrule
    Blog Feedback        & 30 & 5.21E-04 $\pm$ 4.28E-04 & 1 & 5.62E-03 $\pm$ 1.26E-04\\
    Fertility            & 30 & 4.71E-05 $\pm$ 1.34E-05 & 3 & 7.56E-03 $\pm$ 1.82E-03\\
    Forest Fires         & 30 & 5.52E-05 $\pm$ 3.08E-05 & 3 & 7.17E-03 $\pm$ 1.11E-03\\
    QSAR Fish Toxicity   & 25 & 4.16E-05 $\pm$ 1.09E-05 & 4 & 3.66E-03 $\pm$ 1.41E-03\\
    Residential Building & 25 & 3.55E-04 $\pm$ 2.55E-04 & 3 & 6.94E-03 $\pm$ 7.20E-04\\
    Slice Localization   & 30 & 5.95E-04 $\pm$ 4.38E-04 & 3 & 8.83E-03 $\pm$ 1.60E-03\\
    Wine                 & 30 & 6.71E-05 $\pm$ 1.42E-05 & 3 & 7.00E-03 $\pm$ 5.63E-04\\
    Yacht Hydrodynamics  & 30 & 6.75E-05 $\pm$ 4.32E-05 & 3 & 3.51E-03 $\pm$ 3.62E-04\\
    \bottomrule
    \end{tabular}
\end{table}
Approximation errors for all datasets are shown in \cref{tab:errors}.
In the large sample case (train-test split is $5100:51$), setting $d=30$ minimizes the test MSE for most of the datasets (except QSAR Fish Toxicity and Residential Building).
This indicates that, as the approximation theorem (\cref{thm:approx}) implies, the B\'ezier simplex can approximate the solution mapping of the elastic net no matter what dataset is given.
When the size of the training set is 51, setting $d=3$ minimizes the test MSE for most of the datasets (except Blog Feedback and QSAR Fish Toxicity).
This result implies that the optimal degree may be insensitive to what dataset the elastic net fits and mainly determined by the size of the training set for the B\'ezier simplex (i.e., the number of trained elastic net models).

\subsection{Discussion}\label{sec:discussion}
We have seen in \cref{fig:scatter} that to obtain a good approximation, the Pareto set requires a large sample and high degree while the Pareto front does a small sample and low degree.
This may be caused by the difference in the smoothness of the solution mapping.
Due to variable selection by the $L_1$-regularization in the elastic net, the mapping $\theta^*:\Delta^2\to X^*(\tilde f)$ has ``corners'' as shown in the middle plot of \cref{fig:scatter-gt}.
Nevertheless, as shown in the right plot of \cref{fig:scatter-gt}, such ``corners'' disappear in the mapping $\tilde f\circ\theta^*:\Delta^2\to\tilde f(X^*(\tilde f))$.
Similar features have been observed across all the datasets (their scatter plots are omitted due to space limitations).
We expect that the required sample size and degree would be decreased if one can subdivide the B\'ezier simplex at non-smooth points.
The development of such a subdivision algorithm is future work.

To find the best weight on the approximation, weights corresponding to typical hyper-parameters will help for comparison.
If we have some hyper-parameters in the original problem, then transformation from a hyper-parameter $(\mu, \lambda)$ to a weight $(w_1, w_2, w_3)$ gives such ``guiding'' weights.
For any $\mu, \lambda$ such that
\begin{equation}\label{eq:mu}
0 \leq \mu \leq \frac{\lambda - \varepsilon}{\varepsilon},
\end{equation}
the minimizer of the function $g_{\mu, \lambda}$ in \cref{eqn:elastic-net} is the point $\theta^*(w(\mu, \lambda))$, where $w(\mu, \lambda) = (w_1(\mu, \lambda), w_2(\mu, \lambda), w_3(\mu, \lambda))$ is defined by
\begin{equation}\label{eq:elastic-net-sop2mop}
    \begin{split}
        w_1(\mu, \lambda) &= \frac{1 + \varepsilon}{\lambda + \mu + 1},\\
        w_2(\mu, \lambda) &= \frac{(1 + \varepsilon) \mu}{\lambda + \mu + 1},\\
        w_3(\mu, \lambda) &= \frac{\lambda - \varepsilon (\mu + 1)}{\lambda + \mu + 1}.
    \end{split}
\end{equation}
Here, by \cref{eq:mu}, note that $\lambda + \mu + 1 \neq 0$ and $w(\mu, \lambda) \in \Delta^2 \setminus \Delta^2_{\set{2, 3}}$ (see \cref{sec:derivation}, which also contains the derivation of \cref{eq:mu,eq:elastic-net-sop2mop}).

\subsection{Derivation of \texorpdfstring{\cref{eq:mu,eq:elastic-net-sop2mop}}{(7.9) and (7.10)}}\label{sec:derivation}
Let us first derive the equations in \cref{eq:elastic-net-sop2mop} that converts $(\mu, \lambda)$ to $(w_1, w_2, w_3)$, which is true for some superset of $\Delta^2 \setminus \Delta^2_{\set{2, 3}}$.
By $\mu = w_2 / w_1$ in \cref{eq:elastic-net-mop2sop} and $w_1 + w_2 + w_3 = 1$, we have
\begin{equation}\label{eq:w3-1}
\begin{split}
     w_3 &= 1 - w_1 - w_2\\
         &= 1 - w_1 - \mu w_1\\
         &= 1 - (\mu + 1) w_1.
\end{split}
\end{equation}
By $\lambda = (w_3 + \varepsilon) / w_1$ in \cref{eq:elastic-net-mop2sop}, we have
\begin{equation}\label{eq:w3-2}
    w_3 = \lambda w_1 - \varepsilon.
\end{equation}
Combining \cref{eq:w3-1,eq:w3-2}, we have
\begin{align*}
    1 - (\mu + 1) w_1 &= \lambda w_1 - \varepsilon\\
    \iff (\lambda + \mu + 1) w_1 &= 1 + \varepsilon.
\end{align*}
Since $\mu, \lambda \geq 0$, we have $\lambda + \mu + 1 > 0$ and
\begin{equation}\label{eq:w1}
    w_1 = \frac{1 + \varepsilon}{\lambda + \mu + 1}.
\end{equation}
Then, we substitute \cref{eq:w1} to \cref{eq:w3-2} and obtain
\begin{equation}\label{eq:w3}
\begin{split}
    w_3 &= \lambda w_1 - \varepsilon\\
        &= \lambda \times \frac{1 + \varepsilon}{\lambda + \mu + 1} - \varepsilon\\
        &= \frac{\lambda (1 + \varepsilon) - \varepsilon (\lambda + \mu + 1)}{\lambda + \mu + 1}\\
        &= \frac{\lambda - \varepsilon \mu - \varepsilon}{\lambda + \mu + 1}.
\end{split}
\end{equation}
Finally, we have
\begin{align}
    w_2 &= 1 - w_1 - w_3\notag\\
        &= \frac{(\lambda + \mu + 1) - (1 + \varepsilon) - (\lambda - \varepsilon \mu - \varepsilon)}{\lambda + \mu + 1}\notag\\
        &= \frac{(1 + \varepsilon) \mu}{\lambda + \mu + 1}.\label{eq:w2}
\end{align}

Next, let us restrict the possible range of $\mu, \lambda$ to satisfy $w(\mu, \lambda) \in \Delta^2 \setminus \Delta^2_{\set{2, 3}}$, which derives \cref{eq:mu}.
By $0 < w_1 \leq 1$, it follows from \cref{eq:w1} that
\[
    1 + \varepsilon \leq \lambda + \mu + 1,
\]
which can be simplified to
\begin{equation}\label{eq:mu1}
    \varepsilon - \lambda \leq \mu.
\end{equation}
By $0 \leq w_2 \leq 1$, it follows from \cref{eq:w2} that
\[
    (1 + \varepsilon) \mu \leq \lambda + \mu + 1,
\]
which can be simplified to
\[
    \varepsilon \mu \leq \lambda + 1.
\]
Since $\varepsilon > 0$, we have
\begin{equation}\label{eq:mu2}
    \mu \leq \frac{\lambda + 1}{\varepsilon}.
\end{equation}
By $0 \leq w_3 \leq 1$, it follows from \cref{eq:w3} that
\[
    0 \leq \lambda - \varepsilon \mu - \varepsilon \leq \lambda + \mu + 1.
\]
It can be simplified to
\begin{equation}\label{eq:mu3}
    -1 \leq \mu \leq \frac{\lambda - \varepsilon}{\varepsilon}.
\end{equation}
By \cref{eq:mu1,eq:mu2,eq:mu3} and $\mu \geq 0$, we have
\[
    0 \leq \mu \leq \frac{\lambda - \varepsilon}{\varepsilon}.
\]

\section{Conclusions}\label{sec:conclusions}
In this paper, we have shown that all unconstrained strongly convex problems are weakly simplicial.
As an engineering application of the main theorem, we have reformulated the elastic net into a multi-objective strongly convex problem and approximated its solution mapping by a B\'ezier simplex.
Such an approximation enables us to explore the hyper-parameter space of the elastic net within reasonable computation resources.

In future work, we plan to study what subclass of strongly convex problems are simplicial.
While the question has been resolved for $C^1$ mappings \cite{Hamada2019b},
removing its differentiability assumption will bring us to a completely different stage where new mathematics for optimization may be waiting for discovery.

\section*{Acknowledgements}
The authors are grateful to \mbox{Kenta~Hayano}, \mbox{Yutaro~Kabata} and \mbox{Hiroshi~Teramoto} for their kind comments.
\mbox{Shunsuke~Ichiki} was supported by JSPS KAKENHI Grant Numbers JP21K13786 and JP17H06128.
This work is based on the discussions at 2018 IMI Joint Use Research Program, Short-term Joint Research ``Multiobjective optimization and singularity theory: Classification of Pareto point singularities'' in Kyushu University.
This work was also supported by the Research Institute for Mathematical Sciences, a Joint Usage/Research Center located in Kyoto University.
The computation was carried out using the computer resource offered under the category of Intensively Promoted Projects by Research Institute for Information Technology, Kyushu University.

\bibliographystyle{plain}
\bibliography{main}

\begin{thebibliography}{10}

\bibitem{Aeberhard1992}
S.~Aeberhard, D.~Coomans, and O.~de~Vel.
\newblock Comparison of classifiers in high dimensional settings.
\newblock Technical Report 92-02, Dept. of Computer Science and Dept. of
  Mathematics and Statistics, James Cook University of North Queensland., 1992.

\bibitem{Buza2014}
K.~Buza.
\newblock Feedback prediction for blogs.
\newblock In Myra Spiliopoulou, Barbara Schmidt-Thieme, and Ruth Janning,
  editors, {\em Data Analysis, Machine Learning and Knowledge Discovery},
  Studies in Classification, Data Analysis, and Knowledge Organization, pages
  145--152. Springer International Publishing, 2014.

\bibitem{Cassotti2015}
M.~Cassotti, R.~Todeschini D.~Ballabio~and, and V.~Consonni.
\newblock A similarity-based {QSAR} model for predicting acute toxicity towards
  the fathead minnow (pimephales promelas).
\newblock {\em SAR and QSAR in Environmental Research}, 26:217--243, 2015.

\bibitem{Cortez2007}
P.~Cortez and A.~Morais.
\newblock A data mining approach to predict forest fires using meteorological
  data.
\newblock In J.~Neves, M.~F. Santos, and J.~Machado, editors, {\em Proceedings
  of the 13th EPIA 2007 -- Portuguese Conference on Artificial Intelligence},
  New Trends in Artificial Intelligence, pages 512--523, Guimarães, Portugal,
  December 2007. APPIA.

\bibitem{Dua2019}
Dheeru Dua and Casey Graff.
\newblock {UCI} machine learning repository, 2017.

\bibitem{Gil2012}
David Gil, Jose~Luis Girela, Joaquin {De Juan}, M.~Jose Gomez-Torres, and
  Magnus Johnsson.
\newblock Predicting seminal quality with artificial intelligence methods.
\newblock {\em Expert Systems with Applications}, 39(16):12564--12573, 2012.

\bibitem{Graf2011}
Franz Graf, Hans-Peter Kriegel, Matthias Schubert, Sebastian P{\"o}lsterl, and
  Alexander Cavallaro.
\newblock 2{D} image registration in {CT} images using radial image
  descriptors.
\newblock In Gabor Fichtinger, Anne Martel, and Terry Peters, editors, {\em
  Medical Image Computing and Computer-Assisted Intervention -- MICCAI 2011},
  pages 607--614, Berlin, Heidelberg, 2011. Springer Berlin Heidelberg.

\bibitem{Hamada2019}
Naoki Hamada, Kenta Hayano, Shunsuke Ichiki, Yutaro Kabata, and Hiroshi
  Teramoto.
\newblock Topology of {Pareto} sets of strongly convex problems.
\newblock {\em SIAM J. Optim.}, 30(3):2659–2686, 2020.

\bibitem{Hamada2019b}
Naoki Hamada and Shunsuke Ichiki.
\newblock Simpliciality of strongly convex problems.
\newblock {\em to appear in J. Math.\ Soc.\ Japan}.
\newblock \url{https://arxiv.org/abs/1912.09328}.

\bibitem{Hamada2019c}
Naoki Hamada and Shunsuke Ichiki.
\newblock Characterization of the equality of weak efficiency and efficiency on
  convex free disposal hulls.
\newblock {\em preprint}, 2019.
\newblock \url{https://arxiv.org/abs/1910.02867}.

\bibitem{Hoerl1970}
Arthur~E. Hoerl and Robert~W. Kennard.
\newblock Ridge regression: Biased estimation for nonorthogonal problems.
\newblock {\em Technometrics}, 12(1):55--67, 1970.

\bibitem{Kobayashi2019}
Ken Kobayashi, Naoki Hamada, Akiyoshi Sannai, Akinori Tanaka, Kenichi Bannai,
  and Masashi Sugiyama.
\newblock B\'ezier simplex fitting: Describing {Pareto} fronts of simplicial
  problems with small samples in multi-objective optimization.
\newblock In {\em Proceedings of the Thirty-Third {AAAI} Conference on
  Artificial Intelligence}, AAAI-19, pages 2304--2313, 2019.

\bibitem{Miettinen1999}
Kaisa Miettinen.
\newblock {\em Nonlinear Multiobjective Optimization}, volume~12 of {\em
  International Series in Operations Research \& Management Science}.
\newblock Springer-Verlag, GmbH, 1999.

\bibitem{Nesterov2004}
Yurii Nesterov.
\newblock {\em Introductory Lectures on Convex Optimization: A Basic Course}.
\newblock Kluwer Academic Publishers, 2004.

\bibitem{Ortigosa2007}
I.~Ortigosa, R.~Lopez, and J.~Garcia.
\newblock A neural networks approach to residuary resistance of sailing yachts
  prediction.
\newblock In {\em Proceedings of the International Conference on Marine
  Engineering MARINE 2007}, 2007.

\bibitem{Peitz2018}
Sebastian Peitz and Michael Dellnitz.
\newblock A survey of recent trends in multiobjective optimal
  control--surrogate models, feedback control and objective reduction.
\newblock {\em Mathematical and Computational Applications}, 23(2), 2018.

\bibitem{Radulescu2019}
Roxana R{\u{a}}dulescu, Patrick Mannion, Diederik~M. Roijers, and Ann Now{\'e}.
\newblock Multi-objective multi-agent decision making: a utility-based analysis
  and survey.
\newblock {\em Autonomous Agents and Multi-Agent Systems}, 34(1):10, Dec 2019.

\bibitem{Rafiei2016}
Mohammad~Hossein Rafiei and Hojjat Adeli.
\newblock A novel machine learning model for estimation of sale prices of real
  estate units.
\newblock {\em Journal of Construction Engineering and Management},
  142(2):04015066, 2016.

\bibitem{Tanaka2020}
Akinori Tanaka, Akiyoshi Sannai, Ken Kobayashi, and Naoki Hamada.
\newblock Asymptotic risk of {B\'ezier} simplex fitting.
\newblock In {\em Proceedings of the AAAI Conference on Artificial
  Intelligence}, volume~34, pages 2416--2424, Apr. 2020.

\bibitem{Tibshirani1996}
Robert Tibshirani.
\newblock Regression shrinkage and selection via the lasso.
\newblock {\em Journal of the Royal Statistical Society. Series B
  (Methodological)}, 58(1):267--288, 1996.

\bibitem{Wang2011}
Lihui Wang, Amos H.~C. Ng, and Kalyanmoy Deb.
\newblock {\em Multi-Objective Evolutionary Optimisation for Product Design and
  Manufacturing}.
\newblock Springer-Verlag London, 2011.

\bibitem{Zou2005}
Hui Zou and Trevor Hastie.
\newblock Regularization and variable selection via the elastic net.
\newblock {\em Journal of the Royal Statistical Society. Series B (Statistical
  Methodology)}, 67(2):301--320, 2005.

\end{thebibliography}
\end{document}